\documentclass[MSNbibl,number,citesort,seceqn,dvips]{arxbj}
\usepackage{graphicx}
\aid{0}
\volume{19}
\issue{3}
\pubyear{2013}
\firstpage{1047}
\lastpage{1085}
\doi{10.3150/12-BEJ445} %kopijuoti is PTS

\makeatletter
\newcommand{\angler}{\rangle}
\newcommand{\anglel}{\langle}
\newcommand{\rrVert}{\Vert}
\newcommand{\rrvert}{\vert}
\newcommand{\llVert}{\Vert}
\newcommand{\llvert}{\vert}
\newtheorem{theorem}{Theorem}[section]
\newremark{example}[theorem]{Example}
\newremark{remark}[theorem]{Remark}
\newtheorem{lemma}[theorem]{Lemma}
\newtheorem{proposition}[theorem]{Proposition}
\newtheorem{corollary}[theorem]{Corollary}

\def\bsuffix #1{#1}

\newcommand{\eqref}[1]{(\ref{#1})}
\newcommand{\p}{\mathbb{P}}
\renewcommand{\mid}{\vert}
\newcommand{\ess}{\operatorname{ess}}
\newcommand{\sign}{\operatorname{sign}}
\def\F{\mathcal{F}}

\makeatother

\begin{document}
\begin{frontmatter}

\title{BS$\Delta$Es and BSDEs with non-Lipschitz drivers:
Comparison, convergence and~robustness}
\runtitle{BS$\Delta$Es and BSDEs with non-Lipschitz drivers}

\begin{aug}
%%%% inicialai - be tarpu
\author[1]{\fnms{Patrick} \snm{Cheridito}\thanksref{1}\ead[label=e1]{dito@princeton.edu}} \and
\author[2]{\fnms{Mitja} \snm{Stadje}\corref{}\thanksref{2}\ead[label=e2]{m.a.stadje@uvt.nl}}
\runauthor{P. Cheridito and M. Stadje} %% auto
\address[1]{Princeton University,
Princeton, NJ 08544, USA. \printead{e1}}
\address[2]{Department of Econometrics and Operations Research,
Tilburg University and CentER, 5000 LE Tilburg, The
Netherlands. \printead{e2}}
\end{aug}

% HISTORY:
\received{\smonth{11} \syear{2010}}
\revised{\smonth{9} \syear{2011}}

% ABSTRACT
%
\begin{abstract}
We provide existence results and comparison principles for
solutions of backward stochastic difference equations
(BS$\Delta$Es) and then prove convergence of these to solutions
of backward stochastic differential equations (BSDEs) when the
mesh size of the time-discretizaton goes to zero. The
BS$\Delta$Es and BSDEs are governed by drivers
$f^N(t,\omega,y,z)$ and $f(t,\omega,y,z),$ respectively. The
new feature of this paper is that they may be non-Lipschitz in
$z$. For the convergence results it is assumed that the
BS$\Delta$Es are based on $d$-dimensional random walks $W^N$
approximating the $d$-dimensional Brownian motion $W$
underlying the BSDE and that $f^N$ converges to $f$. Conditions
are given under which for any bounded terminal condition $\xi$
for the BSDE, there exist bounded terminal conditions $\xi^N$
for the sequence of BS$\Delta$Es converging to $\xi$, such that
the corresponding solutions converge to the solution of the
limiting BSDE. An important special case is when $f^N$ and $f$
are convex in $z.$ We show that in this situation, the
solutions of the BS$\Delta$Es converge to the solution of the
BSDE for every uniformly bounded sequence $\xi^N$ converging to
$\xi$. As a consequence, one obtains that the BSDE is robust in
the sense that if $(W^N,\xi^N)$ is close to $(W,\xi)$ in
distribution, then the solution of the $N$th BS$\Delta$E is
close to the solution of the BSDE in distribution too.
\end{abstract}

% KEYWORDS
%

\end{frontmatter}

%s1 #&#
\section{Introduction}\label{secint}

The aim of this paper is to obtain general convergence results of
solutions of
stochastic backward equations in discrete time (BS$\Delta$Es) to
solutions of
stochastic backward equations in continuous time
(BSDEs). The discrete equations are governed by drivers $f^N(t,\omega,y,z)$,
$N \in\mathbb{N}$, and the continuous one
by $f(t,\omega,y,z)$. The new feature of this paper is that
$f^N$ and $f$ may be non-Lipschitz in $z$. We assume
that the BS$\Delta$Es are based on $d$-dimensional random walks $W^N$
converging
to the $d$-dimensional Brownian motion $W$ underlying the BSDE
and that $f^N$ tends to $f$. Convergence results for Lipschitz drivers
have been obtained by
Briand \textit{et al.}~\cite{4,5} as well as Toldo~\cite{29,30}. In these papers,
existence and uniqueness of solutions follow from a Picard iteration
argument. Using
results on convergence of filtrations from Coquet \textit{et al.}~\cite{13}, it can
be shown
that the Picard sequences approach each other asymptotically, which
yields general
convergence results. In the case of non-Lipschitz drivers this approach
does not work, and
neither the existence of solutions of BS$\Delta$Es nor their
convergence to their counterparts in
continuous time are clear.

In this paper, we start with a careful analysis of BS$\Delta$Es.
Central to our approach is Theorem~\ref{thmcomp} which provides
a comparison principle for BS$\Delta$Es. It requires drivers
that can grow faster than linearly but strictly less than
quadratically in $z$. Kobylanski~\cite{21} showed existence,
comparison and uniqueness of solutions to BSDEs with general
bounded terminal conditions and drivers of quadratic growth in
$z$. However, in discrete time the situation is different.
Example~\ref{ex2} shows that neither a general comparison
principle nor convergence of solutions for diminishing step
sizes can hold for BS$\Delta$Es if the drivers grow
quadratically in $z$. Our main convergence results are Theorems
\ref{thmone} and~\ref{thmtwo}. Theorem~\ref{thmone} shows that
if $f^N$ and $f$ grow less than quadratically in $z$, then for any
bounded terminal condition $\xi$ for the BSDE, there exist bounded terminal
conditions $\xi^N$ for the BS$\Delta$Es such that the corresponding
solutions $Y^N$
converge to the solution $Y$ of the BSDE in the following sense:
%
%e1.1 #&#
\begin{equation}\label{convYN}
\sup_{0\le t\le T} \bigl|Y^N_t -
Y_t\bigr| \to 0 \qquad \mbox{in } L^2 \mbox{ when } N \to\infty.
\end{equation}
Furthermore, if $\xi$ is of the form $\xi=\varphi(W_{s_1},\ldots
,W_{s_n})$ for a bounded, uniformly
continuous function $\varphi$, then the $\xi^N$ can be chosen
as $\xi^N = \varphi(W^N_{s_1},\ldots,W^N_{s_n})$. In Theorem
\ref{thmtwo}, we prove that if the drivers $f^N$ are convex in
$z$, then \eqref{convYN} holds for every sequence of
uniformly bounded $\xi^N$ converging to $\xi$ in $L^2$. As a
corollary one obtains that if $(W^N, \xi^N)$ is close to
$(W,\xi)$ in distribution, then $Y^N$ is close to $Y$ in
distribution too.

Discrete schemes for the approximation of solutions of BSDEs
have been studied by a number of authors; see for instance, Ma
\textit{et al.}~\cite{22}, Douglas \textit{et al.}~\cite{16}, Bally~\cite{1},
Chevance~\cite{10}, Coquet \textit{et al.}~\cite{12},
Ma \textit{et al.}~\cite{23}, Zhang and Zheng~\cite{32}, Zhang~\cite{31},
Bouchard and Touzi~\cite{3}, Gobet \textit{et al.}~\cite{18} and Otmani~\cite{25}.
However, in all these papers the
drivers are assumed to be Lipschitz. Recently, Imkeller and
Reis~\cite{20} as well as Richou~\cite{28} have obtained results on the
convergence of solutions of discretized BSDEs with drivers of quadratic growth
under regularity assumptions on the terminal conditions and for
specially chosen discrete-time drivers $f^N$.
In Cheridito and Stadje~\cite{9} convergence results are shown for
convex drivers and terminal conditions that are Lipschitz continuous
in the underlying Brownian motion.
Our results hold for general terminal conditions and general drivers
$f^N$ converging to $f$. But they
need subquadratic growth of $f^N$ in~$z$.
Comparison results for BS$\Delta$Es have also been studied in Cohen
and Elliott~\cite{11}
but under different assumptions than here.

The structure of the paper is as follows: In Section~\ref{secnot}, we
introduce the notation and
provide some background material. Then we give an example showing that
BS$\Delta$Es with non-Lipschitz drivers need not converge if the terminal
conditions are not uniformly bounded. In Section~\ref{secsol}, we
show that\vadjust{\goodbreak}
BS$\Delta$Es admit solutions under very mild assumptions if the
time-discretization is fine enough.
Section~\ref{seccomp} starts with an example showing two facts about
BS$\Delta$Es with drivers of quadratic growth: (a) a general comparison
principle cannot hold
and (b) solutions of BS$\Delta$Es can explode if the step-size goes to
zero even if the terminal
conditions are uniformly bounded and converge to zero in $L^2$. We then
prove a general
comparison principle for subquadratic BS$\Delta$Es. Section \ref
{secconvergence}
gives convergence results of solutions of general BS$\Delta$Es to
solutions of BSDEs,
and in Section~\ref{secvex} we prove convergence results for drivers
that are convex in $z$.

%s2 #&#
\section{Notation and setup}\label{secnot}

We fix a finite time horizon $T \in\mathbb{R}_+$. As underlying
process for the BSDE, we take a $d$-dimensional Brownian motion
$(W_t)_{t\in[0,T]}$ on a
probability space $(\Omega, \mathcal{F}, \p)$ and denote by $(\F_t)_{t
\in[0,T]}$ the augmented filtration generated by $(W_t)_{t \in
[0,T]}$. Equalities and inequalities between random variables
will, as usual, be understood in the $\p$-almost sure sense.
As approximating processes we consider a sequence
$(W^N_t)_{t\in[0,T]}$, $N \in\mathbb{N}$, of $d$-dimensional
square-integrable martingales on $(\Omega, \mathcal{F}, \p)$
starting at $0$ with independent
increments satisfying the following three conditions:

\begin{enumerate}[(C2)]
\item[(C1)]
For every $N$ there exists a finite sequence $0=t_0^N < t^N_1< \cdots<
t^N_{i_N}=T$ such that
\[
\lim_{N\rightarrow\infty} \sup_i \bigl|t^N_{i+1} -
t^N_i\bigr| = 0
\]
and $W^N_t$ is constant on each of the intervals $[t^N_i ,t^N_{i+1}).$
\item[(C2)]
\[
\Delta\bigl\anglel W^{N,k}\bigr\angler_{t^N_i} = \Delta\bigl
\anglel W^{N,l}\bigr\angler_{t^N_i} > 0\qquad \mbox{for all } i \mbox{ and } 1 \le k,l \le d .
\]
\item[(C3)]
\[
\lim_{N \to\infty} {\mathbb{E}} \Bigl[\sup_{0\le t\le T} \bigl|W^N_t-W_t\bigr|^2
\Bigr] = 0,
\]
where $|\cdot|$ denotes the standard Euclidean norm on $\mathbb{R}^d\dvt
|x| := (\sum_{i=1}^d x_i^2)^{1/2}$.
\end{enumerate}
Let $(\mathcal{F}^N_t)$ be the filtration generated by $(W^N_t)$ and
define $\anglel W^N\angler_t := \anglel W^{N,k}\angler_t$.
Since $W^N$ has independent increments, $\anglel W^N\angler_t = \anglel W^{N,k}\angler_t$ is equal to
${\mathbb{E}} [(W^{N,k}_t)^2 ]$, and it follows from (C3) that
%
%e2.1 #&#
\begin{eqnarray}
\label{WNt} \sup_{0 \le t \le T} \bigl|\bigl\anglel W^N\bigr
\angler_t - t\bigr| = \sup_{0 \le t \le T} \bigl\llvert {\mathbb{E}} \bigl[
\bigl(W^{N,k}_t\bigr)^2 - \bigl(W^k_t
\bigr)^2 \bigr]\bigr\rrvert \to0\qquad \mbox{for } N \to\infty.
\end{eqnarray}
In particular,
\[
\lim_{N \rightarrow\infty} \max_i \bigl|\Delta\bigl\anglel W^N\bigr
\angler_{t^N_i}\bigr| = 0.
\]

Our standard example for the approximating processes $W^N$ will be
$d$-dimensional Bernoulli random walks.\vadjust{\goodbreak}

%ex2.1 #&#
\begin{example} \label{ExBernoulli}
Let
\[
t^N_i = i \frac{T}{N}\quad \mbox{and}\quad
\tilde{W}^{N,k}_{t^N_i} = \sqrt{\frac{T}{N}} \sum
_{j=1}^{i} X^{N,k}_j
\]
for i.i.d. random variables $X^{N,k}_j$ on a probability space
$(\tilde{\Omega}, \tilde\mathcal{F}, \tilde{\p})$ with
distribution $\tilde{\p}[X^{N,k}_j = \pm1] = 1/2$. Extend
$(\tilde{W}^N_t)$ to $[0,T]$ such that it is constant on the
intervals $[t^N_i, t^N_{i+1})$. Then conditions (C1) and (C2)
are satisfied. To fulfill (C3), one must transfer the random
walks to another probability space. Since they converge to
$d$-dimensional Brownian motion in distribution, there exists a
probability space $(\Omega, \mathcal{F}, \p)$ with a
$d$-dimensional Brownian motion $(W_t)$ and random walks
$(W^N_t)$ having the same distributions as $(\tilde{W}^N_t)$
such that
%
%e2.2 #&#
\begin{equation}
\label{convergence} \sup_{0\le t\le T} \bigl|W^{N}_t-
W_t\bigr| \to0 \quad\mbox{almost surely} \quad\mbox{as } N \to \infty;
\end{equation}
see, for instance, Theorem I.2.7 in Ikeda and
Watanabe~\cite{19}. It can be shown that the sequence
$\sup_{0\le t\le T} |W^{N}_t- W_t|^2$ is uniformly integrable.
Therefore, the convergence \eqref{convergence} also holds in
$L^2$, and condition (C3) is satisfied.
\end{example}

The \textit{driver} of the BSDE is a
$\mathcal{P} \otimes\mathcal{B}(\mathbb{R})\otimes\mathcal
{B}(\mathbb{R}^d)$-measurable
function
\[
f \dvtx  [0,T]\times\Omega\times\mathbb{R}\times\mathbb {R}^d\rightarrow
\mathbb{R},
\]
where $\mathcal{P}$ denotes the predictable $\sigma$-algebra on
$[0,T] \times\Omega$
with respect to $(\mathcal{F}_t)$ and $\mathcal{B}(\mathbb{R})$ and
$\mathcal{B}(\mathbb{R}^d)$ are the Borel $\sigma$-algebras on
$\mathbb{R}$ and $\mathbb{R}^d$, respectively.
We will assume throughout the paper that for fixed
$(t, \omega)$, $f(t,\omega,y,z)$ is continuous in $(y,z)$. Then
$\mathcal{P} \otimes\mathcal{B}(\mathbb{R})\otimes\mathcal
{B}(\mathbb{R}^d)$-measurability of $f$
is equivalent to $(t,\omega) \mapsto f(t,\omega,y,z)$ being
predictable for all fixed $(y,z)$.\looseness=-1

The approximating BS$\Delta$Es have \textit{drivers}
\[
f^N\dvtx  [0,T] \times\Omega\times\mathbb{R} \times\mathbb
{R}^d\rightarrow\mathbb{R}
\]
that are continuous in $(y,z)$, constant on the intervals $(t^N_i, t^N_{i+1}]$
and such that $\omega\mapsto f^N(t^N_{i+1},\omega, y,z)$ is $\mathcal
{F}^N_{t^N_i}$-measurable.
As usual, we henceforth suppress the dependence of $f$ and $f^N$ on
$\omega$ in the notation.

The \textit{terminal conditions} for the BSDE and BS$\Delta$Es are
given by random variables
$\xi$, $\xi^N$ that are measurable with respect to $\mathcal{F}_T$
and $\mathcal{F}^N_T$, respectively.

A \textit{solution} of the BSDE consists of a pair of predictable processes
$(Y_t, Z_t)$ with values in $\mathbb{R} \times\mathbb{R}^d$ such that
\[
{\mathbb{E}} \Bigl[\sup_{0\le t\leq T} Y_t^2 \Bigr] <
\infty,\qquad {\mathbb{E}} \biggl[ \biggl(\int_0^T
|Z_s|^2 \,\mathrm{d}s \biggr)^{1/2} \biggr] < \infty,
\]
and
%
%e2.3 #&#
\begin{equation}
\label{bsde} Y_t = \xi+ \int_t^T
f(s,Y_s,Z_s)\,\mathrm{d}s -\int_t^T
Z_s \,\mathrm{d}W_s,\qquad 0 \le t \le T.\vadjust{\goodbreak}
\end{equation}
In contrast to $(W_t)$, the approximating processes $(W^N_t)$ do in
general not
have the predictable representation property. Therefore, a \textit
{solution} of the
$N$th BS$\Delta$E is a triple of $(\mathcal{F}^N_t)$-adapted processes
$(Y^N_t,Z^N_t,M^N_t)$ taking values in
$\mathbb{R} \times\mathbb{R}^d \times\mathbb{R}$ such that
$(Y^N_t)$ is constant on the intervals $[t^N_i, t^N_{i+1})$,
$(Z^N_t)$ is constant on the intervals $(t^N_i, t^N_{i+1}]$,
$(M^N_t)$ is a square-integrable martingale starting at $0$ and
orthogonal to $(W^N_t)$ that is
constant on the intervals $[t^N_i, t^N_{i+1})$ and
%
%e2.4 #&#
\begin{equation}
\label{bdisc} Y^N_t = \xi^N + \int
_{(t,T]} f^N\bigl(s, Y^N_{s-},
Z^N_s\bigr)\, \mathrm{d} \bigl\anglel W^N\bigr
\angler_s - \int_{(t,T]} Z^N_s\,
\mathrm{d} W^N_s - \bigl(M^N_T -
M^N_t\bigr).
\end{equation}
Due to the particular form of $(Y^N_t, Z^N_t, M^N_t)$, \eqref{bdisc}
is equivalent to
%
%e2.5 #&#
%e2.6 #&#
\begin{eqnarray}
\label{disc} Y^N_{t^N_i} &=& Y^N_{t^N_{i+1}} +
f^N\bigl(t^N_{i+1},Y^N_{t^N_i},
Z^N_{t^N_{i+1}}\bigr) \Delta\bigl\anglel W^N\bigr
\angler_{t^N_{i+1}} - Z^N_{t^N_{i+1}} \Delta
W^N_{t^N_{i+1}} - \Delta M^N_{t^N_{i+1}},
\\
Y^N_T &=& \xi^N.
\end{eqnarray}
Note that if $(W^N_t)$ is a one-dimensional Bernoulli random walk, it
has the
predictable representation property and the orthogonal martingale terms in
\eqref{bdisc} and \eqref{disc} disappear.

It is well known that if the driver $f$ is Lipschitz-continuous
in $(y,z)$ and the terminal condition $\xi$ is in $L^2$, the
BSDE \eqref{bsde} admits a unique solution $(Y,Z)$;
see, for instance, Pardoux and Peng~\cite{26}
or the survey paper by El Karoui \textit{et al.}~\cite{17}. Concerning the
approximation of
BSDEs with Lipschitz drivers, we recall the following result from
Briand \textit{et al.}~\cite{5}.
Their assumptions are slightly different. But the result also holds in
our setup.

%th2.2 #&#
\begin{theorem}[(Briand \textit{et al.}~\cite{5})] \label{thmBriand}
Assume $\xi^N \to\xi$ in $L^2$ and there exists a constant
$K \in\mathbb{R}_+$ such that for all $N \in\mathbb{N}$, $y,y' \in
\mathbb{R}$
and $z,z' \in\mathbb{R}^d$ the following four conditions hold:
\begin{enumerate}[(iii)]
\item[{(i)}] ${\mathbb{E}} [\sup_t f(t,0,0)^2 ] <
\infty$;
\item[{(ii)}] $|f(t,y,z)-f(t,y',z')| \le K(|y-y'|+|z-z'|)$;
\item[{(iii)}] $|f^N(t,y,z) - f^N(t,y',z')| \le K(|y-y'|+|z-z'|)$;
\item[{(iv)}] $\sup_t |f^N(t,y,z)-f(t,y,z)| \rightarrow0 \mbox
{ in } L^2 \mbox{ as } N \to\infty$.
\end{enumerate}
Then, for $N$ large enough, the $N$th BSDE has a unique solution
$(Y^N,Z^N,M^N)$, and
\[
\sup_t \biggl(\bigl|Y^N_t-Y_t\bigr| + \biggl|
\int_0^t Z^N_s
\,\mathrm{d}W^N_s - \int_0^t
Z_s \,\mathrm{d}W_s\biggr|+ \bigl|M^N_t\bigr| \biggr)
\stackrel{(N \rightarrow\infty)} {\to} 0 \qquad \mbox{in } L^2
\]
as well as
\begin{eqnarray*}
&&\sup_t \Biggl(\sum_{k=1}^d
\biggl\llvert \int_0^t Z^{N,k}_s
\,\mathrm{d} \bigl\anglel W^N\bigr\angler_s - \int
_0^t Z^k_s \,\mathrm{d}s\biggr
\rrvert^2 + \biggl\llvert \int_0^t
\bigl|Z^N_s\bigr|^2\, \mathrm{d} \bigl\anglel W^N
\bigr\angler_s - \int_0^t
|Z_s|^2\, \mathrm{d}s\biggr\rrvert \Biggr)\\
&&\quad \stackrel{(N \rightarrow
\infty)} {\to} 0\qquad \mbox{in } L^1,
\end{eqnarray*}
where $(Y,Z)$ is the unique solution of the BSDE \eqref{bsde}.
\end{theorem}

%re2.3 #&#
\begin{remark}
Two special cases of terminal conditions satisfying $\xi^N \to\xi$
in $L^2$ are:
\begin{enumerate}[(b)]

\item[(a)]
$\xi=\varphi(W_T)$ and $\xi^N=\varphi(W^N_T)$ for a continuous function
$\varphi\dvtx  \mathbb{R}^d \to\mathbb{R}$
such that $\varphi^2(W^N_T)$, $N \in\mathbb{N}$, is uniformly integrable.

\item[(b)]
$\xi\in L^2(\F_T)$ general and $\xi^N={\mathbb{E}} [\xi|\F^N_T ]$.
\end{enumerate}
\end{remark}

The aim of this paper is to obtain similar convergence results for
non-Lipschitz drivers.
However, the following example shows that we cannot hope for general
results under the sole
assumption $\xi^N \to\xi$ in $L^2$.

%ex2.4 #&#
\begin{example} \label{ex1}
Consider a one-dimensional Bernoulli random walk with $T=1$, $t_i^N = i/N$
and $\p[\Delta W^N_{t^N_i} = \pm\sqrt{1/N}] = 1/2$. Then
\[
\Delta\bigl\anglel W^N\bigr\angler_{t_i^N} = {\mathbb{E}}
\bigl[\bigl(\Delta W^N_{t^N_i}\bigr)^2 \bigr] =
1/N.
\]
Fix $q \in(1,2)$ and a sequence of constants $a^N \ge2 N^{({1 -
q/2})/({q-1})}$.
Consider the BS$\Delta$Es
\[
Y^N_{t_i^N} = Y^N_{t_{i+1}^N} +
\bigl|Z_{t^N_{i+1}}^N\bigr|^q \Delta\bigl\anglel
W^N\bigr\angler_{t^N_{i+1}} - Z^N_{t^N_{i+1}}
\Delta W^N_{t^N_{i+1}}, \qquad
Y^N_T = a^N 1_{ \{W^N_{t_N} = \sqrt{N} \}}.
\]
It can easily be checked that
\[
Z^N_{t^N_N} = \frac{\sqrt{N}}{2} a^N
1_{ \{W^N_{t_{N-1}^N} =
({N-1})/{\sqrt{N}} \}}\quad \mbox{and}\quad Y^N_{t_{N-1}^N} =
a^N_{t_{N-1}^N} 1_{ \{W^N_{t_{N-1}^N} =
({N-1})/{\sqrt{N}} \}}
\]
for
\[
a^N_{t_{N-1}^N} = \frac{a^N}{2} + 2^{-q}
N^{q/2-1} \bigl(a^N\bigr)^q \ge a^N.
\]
Continuing this way one gets
\[
Z^N_{t^N_{N-1}} = \frac{\sqrt{N}}{2} a^N_{t^N_{N-1}}
1_{ \{
W^N_{t_{N-2}^N} = ({N-2})/{\sqrt{N}} \}} \quad\mbox{and} \quad Y^N_{t_{N-2}^N} =
a^N_{t_{N-2}^N} 1_{ \{W^N_{t_{N-2}^N} = ({N-2})/{\sqrt{N}} \}}
\]
with
\[
a^N_{t_{N-2}^N} = \frac{a^N_{t_{N-1}}}{2} + 2^{-q}
N^{q/2 -1} \bigl(a^N_{t_{N-1}^N} \bigr)^q \ge
a^N_{t_{N-1}^N},
\]
and so on. In particular,
\[
Y^N_0 \ge a^N \ge2 N^{({1 - (q/2)})/({q-1})} \to
\infty\qquad\mbox{for } N \to\infty.
\]
Note that for $a^N = 2 N^{({1 - q/2})/({q-1})}$, one has
$\xi^N \to0$ in $L^p$ for all $p \in(0,\infty)$ but not in~$L^{\infty}$.
\end{example}

The example shows that in the case of super-linear growth
of $f^N$ in $z$ one cannot expect convergence of the discrete-time
solutions if the terminal conditions are uniformly $L^p$-bounded and
converge in $L^p$ for $p < \infty$.
This is not unexpected since in the literature on BSDEs with non-Lipschitz
drivers it is usually required that the terminal condition be in
$L^{\infty}$
or sufficiently well exponentially integrable (see Kobylanski~\cite{21}, or
Briand and Hu~\cite{6}). Consequently, in this paper, we will always assume:
\begin{enumerate}[(C4)]
\item[(C4)]
\[
\sup_N \bigl\llVert \xi^N \bigr\rrVert_{\infty} <
\infty\quad\mbox{and}\quad \llVert \xi\rrVert_{\infty} <\infty.
\]
\end{enumerate}
We shortly summarize the notation and assumptions that have been
introduced in this section:
\begin{itemize}
\item $W^N$, $N \in\mathbb{N}$, is a sequence of discrete-time
martingales approximating the $d$-dimensional
Brownian motion $W$.
\item $f$ and $\xi$ are the driver and terminal condition of the BSDE
\eqref{bsde}. A solution to
\eqref{bsde} will be denoted by $(Y,Z)$.
\item $f^N$ and $\xi^N$ are the drivers and terminal conditions of the
BS$\Delta$Es \eqref{bdisc}.
Solutions will be denoted by $(Y^N,Z^N,M^N)$.
\item We always assume (C1)--(C4).
\end{itemize}

%s3 #&#
\section{\texorpdfstring{Solutions of BS$\Delta$Es}{Solutions of BS Delta Es}}\label{secsol}

In this section, we present two results on solutions of BS$\Delta$Es
that will be needed later in the paper.
Their proofs are straightforward and therefore, given in the \hyperref[app]{Appendix}.

%le3.1 #&#
\begin{lemma} \label{lemmaYZM}
If a solution $(Y^N,Z^N, M^N)$ of the $N$th BS$\Delta$E exists, one has
%
%e3.1 #&#
\begin{equation}
\label{Yformula} Y^N_{t^N_i} - f^N
\bigl(t^N_{i+1},Y^N_{t^N_i},Z^N_{t^N_{i+1}}
\bigr) \Delta \bigl\anglel W^N\bigr\angler_{t^N_{i+1}}={
\mathbb{E}} \bigl[Y^N_{t^N_{i+1}}|\F^N_{t^N_i}
\bigr],
\end{equation}
and the pair $(Z^N, M^N)$ is uniquely determined by $Y^N$
through
%
%e3.2 #&#
%e3.3 #&#
\begin{eqnarray}
\label{Zformula} Z^{N,k}_{t^N_{i+1}} &=& \frac{{\mathbb{E}} [Y^N_{t^N_{i+1}} \Delta
W^{N,k}_{t^N_{i+1}}|\F^N_{t^N_i} ]}{\Delta\anglel
W^{N}\angler_{t^N_{i+1}}},
\\
\label{Mformula} \Delta M^N_{t^N_{i+1}} &=& Y^N_{t^N_{i+1}}
- {\mathbb{E}} \bigl[Y^N_{t^N_{i+1}}|\F^N_{t^N_i}
\bigr] - Z^N_{t^N_{i+1}} \Delta W^N_{t^N_{i+1}}.
\end{eqnarray}
\end{lemma}

Concerning the existence of solutions to BS$\Delta$Es, one has
the following result. For the special case where $W^N$ is a
one-dimensional Bernoulli random walk, see Peng~\cite{27}.

%pr3.2 #&#
\begin{proposition} \label{propex}
Assume there exists a constant $K \in\mathbb{R}_+$ and a locally bounded
function $g \dvtx  \mathbb{R}^d \to\mathbb{R}_+$ such that
\[
\bigl|f^N(t,y,z)\bigr| \le K\bigl(1+|y| + g(z)\bigr)\quad \mbox{and}\quad
\max_i \Delta\bigl\anglel W^N\bigr\angler_{t^N_i}
< 1/K .
\]
Then the $N$th BS$\Delta$E has a solution $(Y^N,Z^N,M^N)$ such that
$Y^N$ and $Z^N$ are bounded.
If $W^N$ is bounded, then so is $M^N$.\vadjust{\goodbreak}
\end{proposition}

%re3.3 #&#
\begin{remark}
For $\max_i \Delta\anglel W^N\angler_{t^N_i} \ge1/K$ a solution of
the $N$th
BS$\Delta$E might not exist. For example, let $W^1$ be a
one-dimensional Bernoulli random walk with
$t_0^1=0$, $t_1^1=1 = T$, $\p[\Delta W^1_1 = \pm1] =1/2$, $\xi^1=1$
and $f^1(t,y,z)=y$.
Since the terminal condition is deterministic, one must choose $Z^1_1 = 0$,
and (\ref{Ysolvi}) becomes
\[
Y^1_0 - Y^1_0 = 1,
\]
an equation without solution.
\end{remark}

%s4 #&#
\section{\texorpdfstring{Comparison principle for BS$\Delta$Es}{Comparison principle for BS Delta Es}}\label{seccomp}

Our main tool to derive convergence results
will be a comparison principle for BS$\Delta$Es of the
following form: Let $f^N_1$, $f^N_2$ be drivers and $\xi^N_1$,
$\xi^N_2$ terminal conditions such that $f^N_1(t,y,z) \ge
f^N_2(t,y,z)$ for all $t,y,z$ and $\xi^N_1 \ge\xi^N_2$. Then
the corresponding solutions satisfy $Y^N_{1,t} \ge Y^N_{2,t}$
for all $t$.

The next example shows that if the drivers grow quadratically
in $z$, a general comparison principle for BS$\Delta$Es cannot hold.

%ex4.1 #&#
\begin{example} \label{ex2}
As in Example~\ref{ex1}, let $W^N$ be a one-dimensional
Bernoulli random walk with $T=1$, $t^N_i = i/N$ and $\p[\Delta W^N_{t^N_i}
= \pm\sqrt{1/N}] = 1/2$. Consider the BS$\Delta$Es
%
%e4.1 #&#
%e4.2 #&#
\begin{eqnarray}
\label{bsdep=2} Y^N_{t^N_i} &=& Y^N_{t^N_{i+1}}
+ \bigl(Z^N_{t^N_{i+1}}\bigr)^2 \Delta\bigl\anglel
W^N\bigr\angler_{t^N_{i+1}} - Z^N_{t^N_{i+1}}
\Delta W^N_{t^N_{i+1}},
\\
Y^N_T &=& a 1_{ \{W^N_{t_N^N} = \sqrt{N} \}}
\end{eqnarray}
for a constant $a > 2$ and define $\varepsilon> 0$ by $a = 2(1+ 2
\varepsilon)$.
Then
\[
Z^N_{t^N_N} = \frac{\sqrt{N}}{2} a 1_{ \{W^N_{t_{N-1}^N} =
({N-1})/{\sqrt{N}} \}} ,\qquad
Y^N_{t_{N-1}^N} = a^N_{t_{N-1}^N}
1_{ \{W^N_{t_{N-1}^N} =
({N-1})/{\sqrt{N}} \}},
\]
where
\begin{eqnarray*}
a^N_{t_{N-1}^N} = \frac{a}{2} + \biggl(\frac{a}{2}
\biggr)^2 = a (1 + \varepsilon)
\end{eqnarray*}
and
\[
Z^N_{t^N_{N-1}} = \frac{\sqrt{N}}{2} a^N_{t_{N-1}^N}
1_{ \{W^N_{t_{N-2}^N} = ({N-2})/{\sqrt
{N}} \}} ,\qquad Y_{t_{N-2}^N}^N = a^N_{t_{N-2}^N}
1_{ \{W^N_{t_{N-2}^N} =
({N-2})/{\sqrt{N}} \}}
\]
for
\begin{eqnarray*}
&& a^N_{t_{N-2}} = \frac{a^N_{t_{N-1}}}{2} + \biggl(
\frac
{a^N_{t_{N-1}}}{2} \biggr)^2 \ge\frac{a^N_{t_{N-1}}}{2} \biggl(1 +
\frac{a}{2} \biggr) = a^N_{t_{N-1}} (1 + \varepsilon).
\end{eqnarray*}
Continuing this computation, one obtains
\[
Y_0^N \ge a (1+\varepsilon)^N \to\infty\qquad
\mbox{as } N \to\infty.
\]
Note that the terminal conditions $Y_T^N$ are uniformly
$L^{\infty}$-bounded in $N$ and $Y_T^N \to0$ in $L^p$ for all $p <
\infty$. But
the solutions $Y^N_t$ explode as $N \to\infty$. We also point out that
for fixed $N$, the solutions to equation \eqref{bsdep=2} are not
monotone in the terminal
condition. Indeed, $(\tilde{Y}^N_t, \tilde{Z}^N_t) \equiv(a,0)$ is a
solution of
equation \eqref{bsdep=2} with terminal condition $\tilde{Y}^N_T = a
\ge Y^N_T$.
However, $\tilde{Y}^N_0 < Y^N_0$. In particular, the comparison
principle is violated.
\end{example}

In view of Example~\ref{ex2}, we restrict ourselves in the next
theorem to
drivers that grow less than quadratically in $z$. We need the following
assumption on the increments of~$W^N$:
\begin{enumerate}[(W1)]
\item[(W1)]
There exists a constant  $q \in[1,2)$ such that  $\lim_{N \to\infty} \max_{i , k}
\frac{\Vert\Delta W^{N,k}_{t^N_i}\Vert_{\infty}}{\Delta\anglel W^{N}\angler_{t^N_{i}}^{q/4}}= 0$.
\end{enumerate}
Note that the standard Bernoulli random walks of Example \ref
{ExBernoulli} satisfy (W1) for all $q \in[1,2)$.
The subsequent theorem establishes a comparison result for BS$\Delta
$Es governed by non-Lipschitz drivers.

%th4.2 #&#
\begin{theorem} \label{thmcomp}
Let $C, K, L \in\mathbb{R}_+$ and assume \emph{(W1)} holds for some $q
\in[1,2)$.
Then there exists $N_0 \in\mathbb{N}$ such that for every $N \ge N_0$,
all drivers $f^N_1 \ge f^N_2$ and terminal conditions $\xi^N_1 \ge\xi^N_2$
satisfying
\begin{enumerate}[(iii)]
\item[(i)] $\llVert \xi^N_m \rrVert_{\infty} \le C$,
\item[(ii)]
$|f^N_m(t,y,z)| \le K(1+|y|+|z|^q)$ for all $(t,y,z) \in[0,T]
\times\mathbb{R}^{d+1}$,
\item[(iii)]
$|f^N_m(t,y_1,z) - f^N_m(t,y_2,z)| \le L(1+|z|^q) |y_1-y_2|$
for all $(t,y_1, y_2, z) \in[0,T] \times\mathbb{R}^{d+2}$ such that
$|y_1|, |y_2| \le(C+1) \exp(K T)$,
\item[(iv)]
$|f^N_m(t,y,z_1)-f^N_m(t,y,z_2)| \le L(1+(|z_1| \vee|z_2|)^{q/2})
|z_1-z_2|$
for all $(t,y,z_1,z_2) \in[0,T] \times\mathbb{R}^{2d +1}$ such that
$|y| \le(C+1) \exp(K T)$,
\end{enumerate}
the BS$\Delta$Es with parameters $(f^N_m, \xi^N_m)$ have unique solutions
$(Y^N_m, Z^N_m, M^N_m)$, $m = 1,2$, and
\[
(C+1) \exp\bigl(K(T-t)\bigr) \ge Y^N_{1,t} \ge
Y^N_{2,t} \ge - (C+1) \exp\bigl(K(T-t)\bigr)\qquad \mbox{for all } t \in[0,T].
\]
\end{theorem}

To prove Theorem~\ref{thmcomp}, we need the following two lemmas,
whose proofs can be found in the \hyperref[app]{Appendix}.
The first one provides a comparison principle under stronger
assumptions than Theorem~\ref{thmcomp}. The second one
gives conditions under which the $Y^N$ are uniformly bounded in $N$.

%le4.3 #&#
\begin{lemma} \label{lemmacomp1}
Let $C,K \in\mathbb{R}_+$ and assume \emph{(W1)} holds for some $q \in[1,2)$.
Then there exists $N_0 \in\mathbb{N}$ such that for every $N \ge N_0$,
all drivers $f^N_1 \ge f^N_2$ and terminal conditions $\xi^N_1 \ge\xi^N_2$
satisfying conditions \emph{(i)} and \emph{(ii)} of Theorem \ref
{thmcomp} as well as
\begin{enumerate}[(iii)]
\item[(iii)]
$|f^N_m(t,y_1,z) - f^N_m(t,y_2,z)| \le K(1+|z|^q) |y_1-y_2|$
for all $(t,y_1,y_2,z) \in[0,T] \times\mathbb{R}^{d+2}$,\vspace*{1pt}
\item[(iv)]
$|f^N_m(t,y,z_1)-f^N_m(t,y,z_2)| \le q K(1+(|z_1| \vee|z_2|)^{q/2}) |z_1-z_2|$
for all $(t,y,z_1,z_2) \in[0,T] \times\mathbb{R}^{2d+1}$,\vadjust{\goodbreak}
\end{enumerate}
the BS$\Delta$Es with parameters $(f^N_m, \xi^N_m)$ have unique solutions
$(Y^N_{m}, Z^N_{m}, M^N_{m})$, $m=1,2$, and
%
%e4.3 #&#
\begin{equation}
\label{statlemma} (C+1) \exp\bigl(K(T-t)\bigr) \ge Y^N_{1,t}
\ge Y^N_{2,t} \ge- (C+1) \exp\bigl(K(T-t)\bigr) \qquad\mbox{for
all } t \in[0,T].
\end{equation}
\end{lemma}

%le4.4 #&#
\begin{lemma} \label{lemmacomp2}
Let $C,K \in\mathbb{R}_+$ and assume \emph{(W1)} holds for some $q \in[1,2)$.
Then there exists $N_0 \in\mathbb{N}$ such that for every $N \ge N_0$,
all drivers $f^N$ and terminal conditions $\xi^N$
satisfying
\begin{enumerate}[(ii)]
\item[(i)]
$\llVert \xi^N \rrVert_{\infty} \le C$,
\item[(ii)]
$|f^N(t,y,z)| \le K(1+|y|+|z|^q)$ for all $t \in[0,T]$, $y \in\mathbb{R}$
and $z \in\mathbb{R}^d$,
\end{enumerate}
every solution $(Y^N,Z^N,M^N)$ of the $N$th BS$\Delta$E satisfies
%
%e4.4 #&#
\begin{equation}
\label{statlemma2} \bigl|Y^N_t\bigr| \le(C+1) \exp\bigl(K(T-t)
\bigr) \qquad\mbox{for all } t \in[0,T].
\end{equation}
\end{lemma}

We now are ready for the proof.

\begin{pf*}{Proof of Theorem~\ref{thmcomp}}
It follows from Proposition~\ref{propex} and
Lemma~\ref{lemmacomp2} that there exists an $N_1$ such that for
all $N \ge N_1$, the $N$th BS$\Delta$E has a solution $(Y^N,Z^N,M^N)$
for all
$f^N$ and $\xi^N$ satisfying conditions (i) and (ii) of Theorem
\ref{thmcomp}, and every such solution satisfies $|Y^N_t| \le(C+1)
\exp(K(T-t))$, $0 \le t \le T$. Now choose $N_0 \ge N_1$
such that Lemma~\ref{lemmacomp1} holds for $\tilde{K} = K \vee L$
instead of $K$ and fix $N \ge N_0$. If $f^N_1 \ge f^N_2$ and
$\xi^N_1 \ge\xi^N_2$ are drivers and terminal conditions satisfying
conditions (i)--(iv) of Theorem~\ref{thmcomp}, then
there exist corresponding solutions $(Y^N_m, Z^N_m, M^N_m)$,
$m=1,2$, both of which satisfy $|Y^N_{m,t}| \le(C+1) \exp(K(T-t))$.
So one can change the drivers $f^N_m$ for $|y| >
(C+1) \exp(K T)$ such that they satisfy the conditions of
Lemma~\ref{lemmacomp1}, and it follows that $Y^N_{1,t} \ge
Y^N_{2,t}$. In particular, both solutions are unique.
\end{pf*}

%s5 #&#
\section{Convergence results for drivers with subquadratic
growth}\label{secconvergence}

With a slight abuse of notation, the discrete-time drivers can be
written as $f^N(t,W^N,y,z)$.
By predictability, $f^N(t^N_{i+1},W^N, y,z)$ only depends on
$W^N_{t^N_1}, \ldots, W^N_{t^N_i}$.
Let $q \in[1,2)$ and consider the following conditions on the drivers $f^N\dvt$
There exists a constant $K > 0$ such that
\begin{enumerate}[(f2)]
\item[(f1)]
For all $N \in\mathbb{N},$
$w \in\mathbb{R}^{d \times i_N}$
and $(t,y,z) \in[0,T] \times\mathbb{R}^{d+1}$,
\[
\bigl|f^N(t,w,y,z)\bigr|\leq K\bigl(1+|y|+|z|^{q}\bigr).
\]
\item[(f2)] For all $N \in\mathbb{N},$
$w \in\mathbb{R}^{d \times i_N}$
and $(t,y_1,y_2,z)
\in[0,T] \times\mathbb{R}^{d+2}$,
\[
\bigl|f^N(t,w,y_1,z)-f^N(t,w,y_2,z)\bigr|
\le K|y_1-y_2|.
\]
\item[(f3)]
For every $a \in\mathbb{R}_+$ there exists $b \in\mathbb{R}_+$ such
that for all $N \in\mathbb{N}$, $t \in[0,T]$,
$y \in[-a,a],$ $w \in\mathbb{R}^{d \times i_N}$ and $z_1,z_2 \in
\mathbb{R}^d$,
\[
\bigl|f^N(t,w,y,z_1)-f^N(t,w,y,z_2)\bigr|
\le b\bigl(1+\bigl(|z_1| \vee|z_2|\bigr)^{q/2}\bigr)
|z_1-z_2|.
\]
\item[(f4)]
For all $N \in\mathbb{N}$, $i=0, \ldots,i_N-1$, $w_1, w_2 \in\mathbb
{R}^{d \times i_N}$ and
$(y,z) \in\mathbb{R}^{d+1}$,
\[
\bigl|f^N\bigl(t^N_{i+1},w_1,y,z
\bigr)-f^N\bigl(t^N_{i+1},w_2,y,z
\bigr)\bigr| \leq K \sup_{0 \le t \le t^N_{i}}\bigl| w_1(t) - w_2(t)\bigr|.
\]
\item[(f5)]
For all $(y,z) \in\mathbb{R}^{d+1}$,
\[
\sup_{0\leq t\leq T}\bigl|f^N(t,y,z)-f(t,y,z)\bigr|\rightarrow0 \qquad \mbox{in
$L^2$ as $N \to\infty$.}
\]
\end{enumerate}
For a measurable function $g \dvtx  [0,T] \times\Omega\times
\mathbb{R}^{d+1} \to\mathbb{R}$, denote
\[
\Vert g\Vert_\infty= \mathop{\ess\sup}_{\omega}
\sup_{t,y,z}\bigl|g(t,\omega,y,z)\bigr|.
\]
The following lemma shows that the solutions of the BS$\Delta$Es
are stable in the terminal condition and the driver function.
The proof relies on Theorem~\ref{thmcomp} and can be found in the \hyperref[app]{Appendix}.

%le5.1 #&#
\begin{lemma} \label{infterm2}
Let $C,K \in\mathbb{R}_+$ and assume condition \emph{(W1)} holds for
some $q \in[1,2)$. Then there exists
$N_0 \in\mathbb{N}$ and a constant $D \in\mathbb{R}_+$ such
that for all $N \ge N_0$, all terminal conditions $\xi^N_1$,
$\xi^N_2$ bounded by $C$ and drivers $f^N_1$, $f^N_2$
satisfying \emph{(f1)--(f3)} as well as $\Vert f^N_1-f^N_2\Vert_\infty
\le K$, the BS$\Delta$Es with parameters $(f^N_m, \xi^N_m)$
have unique solutions $(Y^N_{m}, Z^N_{m}, M^N_{m})$, $m = 1,2$, and
\[
\sup_{0\leq t\leq T}\bigl|Y^{N}_{1,t}-Y^{N}_{2,t}\bigr|
\le D \bigl(\bigl\Vert f^N_1-f^N_2\bigr\Vert_\infty+
\bigl\Vert\xi^N_1-\xi^N_2\bigr\Vert_\infty
\bigr).
\]
\end{lemma}

The next lemma shows that for Lipschitz-continuous terminal conditions, the
$Z^N$ are uniformly bounded. This will be a key ingredient in the
proofs of our convergence results.
The proof is given in the \hyperref[app]{Appendix}.

%le5.2 #&#
\begin{lemma} \label{l3b}
Assume \emph{(W1)} and \emph{(f1)--(f4)} hold for some $q \in[1,2)$ and
the $\xi^N$
are of the form $\xi^N= \varphi(W^N_{s_{1}},\ldots, W^N_{s_{n}})$ for fixed
$n \in\mathbb{N}$, $0\le s_{1}< \cdots<s_{n}\le T,$ and
a bounded Lipschitz-continuous function $\varphi\dvtx  \mathbb{R}^{d
\times n} \to\mathbb{R}$.
Then there exists an $N_0 \in\mathbb{N}$ such for all $N \ge N_0$,
the $N$th BS$\Delta$E
has a unique solution $(Y^N,Z^N)$ and
$\sup_{N \ge N_0} \llVert \sup_t|Z^N_t| \rrVert_\infty<
\infty$.\looseness=-1
\end{lemma}

%re5.3 #&#
\begin{remark} \label{remark1}
In general $\sup_{N \ge N_0} \llVert \sup_t |Z^N_t| \rrVert_\infty
< \infty$
does not hold
if $\varphi$ is not Lipschitz-continuous. For example,
consider one-dimensional Bernoulli random walks $W^N$ with $T=1$,
$t^N_i = i/N$ and $\p[\Delta W^N_{t^N_i} = \pm\sqrt{1/N}] = 1/2$.
Let the terminal conditions be of the form
\[
\xi^N = \cases{ \sqrt{W^N_1}\wedge1,
&\quad if $W^N_1 \ge0$,
\cr
-\sqrt{-W^N_1}
\vee-1, &\quad if $W^N_1 < 0$. }\vadjust{\goodbreak}
\]
On the set $\{W^N_{(N-1)/N}\!=\!0\}$ one has
$\xi^N\!=\!{\sign}(\Delta W^N_1)\sqrt{|\Delta W^N_1|}$, and hence, by Lemma~\ref{lemmaYZM},
\[
Z^{N}_1 = \frac{{\mathbb{E}} [\xi^N \Delta W^N_1 |
W^N_{(N-1)/N}=0 ]}{\Delta\anglel W^{N}\angler_1} = N^{1/4}.
\]
In particular, $Z^{N}_1 \rightarrow\infty$
as $N \to\infty$ on the set $\{W^N_{(N-1)/N}=0\}.$\vspace*{-1pt}
\end{remark}

Before we prove convergence of solutions of BS$\Delta$Es to solutions
of BSDEs,
we recall the following result on quadratic BSDEs,
which follows from Theorems 2.5--2.7 of Morlais~\cite{24}.\vspace*{-1pt}

%th5.4 #&#
\begin{theorem} [(Morlais~\cite{24})]\label{Morlais}
Let $K \in\mathbb{R}_+$ such that
%
%e5.1 #&#
%e5.2 #&#
\begin{eqnarray}
\label{f1}  \bigl|f(t,y,z)\bigr| &\le& K\bigl(1+|y|+|z|^{2}\bigr),
\\[-2pt]
\label{f2}  \bigl|f(t,y_1,z)-f(t,y_2,z)\bigr| &\le&
K|y_1-y_2| \qquad\mbox{for all } y_1,
y_2\in \mathbb{R},
\end{eqnarray}
and for every $a \in\mathbb{R}_+$ there exists
$b \in\mathbb{R}_+$ such that
%
%e5.3 #&#
\begin{equation}\label{f3} \bigl|f(t,y,z_1)-f(t,y,z_2)\bigr| \le b
\bigl(1+\bigl(|z_1| \vee|z_2|\bigr)\bigr) |z_1-z_2|
\end{equation}
for all $t \in[0,T]$, $y \in[-a,a]$ and $z_1,z_2 \in\mathbb{R}^d$.
Then the BSDE \eqref{bsde} has a unique solution $(Y,Z)$ such that
$Y$ is bounded. Furthermore, for bounded terminal conditions $\xi_1\!\geq\!\xi_2$ and drivers
$f_1 \!\ge\! f_2$ fulfilling \eqref{f1}--\eqref{f3}, the corresponding
solutions satisfy $Y_{1,t} \!\ge\! Y_{2,t}$ for all~$t$.\looseness=-1\vspace*{-1pt}
\end{theorem}

%re5.5 #&#
\begin{remark}
Actually, Morlais~\cite{24} makes slightly different assumptions. In her
paper, the
underlying noise process is continuous but does not have to be a
Brownian motion,
and condition \eqref{f3} is assumed to hold for a constant $b$
independent of $a$.
However, existence of a solution $(Y,Z)$ with bounded $Y$ already
follows from
\eqref{f1}, and if $Y$ is bounded by a constant $a \in\mathbb
{R}_+$, the
driver $f(t,y,z)$ only matters for $y \in[-a,a]$ and can be modified
so that it
satisfies \eqref{f3} for a constant $b$ independent of $a$. Hence, assumptions
\eqref{f1}--\eqref{f3} are sufficient for Theorem~\ref{Morlais}.\vspace*{-1pt}
\end{remark}

%pr5.6 #&#
\begin{proposition} \label{PropLipschconv}
Assume there exists a $q \in[1,2)$ such that \emph{(W1)} and \emph{(f1)--(f5)} hold. If
$\xi$ and $\xi^N$ are of the form $\xi= \varphi(W_{s_{1}},\ldots
,W_{s_{n}})$ and
$\xi^N= \varphi(W^N_{s_{1}},\ldots, W^N_{s_{n}})$ for fixed
$n \in\mathbb{N}$, $0\le s_{1}<\cdots<s_{n}\le T$, and a bounded
Lipschitz-continuous function
$\varphi\dvtx  \mathbb{R}^{d \times n} \to\mathbb{R}$, then there exists
an $N_0 \in\mathbb{N}$ such that for all $N \ge N_0$,
the $N$th BS$\Delta$E has a unique solution $(Y^N,Z^N,M^N)$ satisfying
$\sup_{N \ge N_0} \llVert \sup_t |Z^N_t| \rrVert_{\infty} < \infty
$, the BSDE
\eqref{bsde} has a
unique solution $(Y,Z)$ with bounded $Y$, and
\[
\sup_t \biggl(\bigl|Y^N_t-Y_t\bigr| + \biggl|
\int_0^t Z^N_s
\,\mathrm{d}W^N_s - \int_0^t
Z_s \,\mathrm{d}W_s\biggr|+ \bigl|M^N_t\bigr| \biggr)
\stackrel{(N\rightarrow\infty)} {\to} 0\qquad \mbox{in } L^2
\]
as well as
\begin{eqnarray*}
&&\sup_t \Biggl(\sum_{k=1}^d
\biggl\llvert \int_0^t Z^{N,k}_s
\,\mathrm{d} \bigl\anglel W^N\bigr\angler_s - \int
_0^t Z^k_s \,\mathrm{d}s\biggr
\rrvert^2 + \biggl\llvert \int_0^t
\bigl|Z^N_s\bigr|^2 \,\mathrm{d} \bigl\anglel W^N
\bigr\angler_s - \int_0^t
|Z_s|^2 \,\mathrm{d}s\biggr\rrvert \Biggr) \\[-2pt]
&&\quad\stackrel{(N\rightarrow
\infty)} {\to} 0 \qquad\mbox{in } L^1.\vadjust{\goodbreak}
\end{eqnarray*}
In particular, there exists a constant $R \in\mathbb{R}_+$ such that
$|Z| \le R$ $\nu\otimes\p$-almost everywhere, where $\nu$ denotes
Lebesgue measure on $[0,T]$.
\end{proposition}

\begin{pf}
It follows from (f1)--(f5) that the driver $f$ satisfies \eqref
{f1}--\eqref{f3}. So one obtains from Theorem~\ref{Morlais}
that the BSDE \eqref{bsde} has a unique solution $(Y,Z)$ such that $Y$
is bounded.
By Lemma~\ref{l3b}, there exists $N_0 \in\mathbb{N}$ such that for
all $N \ge N_0$, the
$N$th BS$\Delta$E has a unique solution $(Y^N,Z^N,M^N)$ and
$\sup_{N \ge N_0} \llVert \sup_t |Z^N_t| \rrVert_\infty\le R$ for
some constant
$R \in\mathbb{R}_+$. Define
\[
\tilde{f}^N(t,y,z) = \cases{ f^N(t,y,z), &\quad for $|z|\leq
R$,
\cr
f^N\bigl(t,y,Rz/|z|\bigr), &\quad for $|z| >R$ }
\]
and
\[
\tilde{f}(t,y,z) = \cases{ f(t,y,z), &\quad for $|z|\le R$,
\cr
f\bigl(t,y,Rz/|z|\bigr), &\quad for $|z|
>R$. }
\]
Then the $\tilde{f}^N$ are uniformly Lipschitz in $(y,z)$ and
\[
\sup_{0 \leq t \leq T}\bigl| \tilde{f}^N(t,y,z)-\tilde{f}(t,y,z)\bigr| \to0\qquad
\mbox{in $L^2$ as $N \to\infty$}.
\]
So it follows that $\tilde{f}^N$ and $\tilde{f}$ fulfill the
conditions of Theorem
\ref{thmBriand}. Denote by $(\tilde{Y}^N,\tilde{Z}^N,\tilde{M}^N)$
the solution
to the $N$th BS$\Delta$E with parameters $(\tilde{f}^N, \xi^N)$ and by
$(\tilde{Y},\tilde{Z})$ the solution of the BSDE corresponding to
$(\tilde{f},\xi)$. Since the $Z^N$ are bounded by $R$,
$(Y^N,Z^N,M^N)$ is also a solution of the BS$\Delta$E
corresponding to $(\tilde{f}^N,\xi^N)$. So it follows from Theorem
\ref{thmcomp} that for $N$ large enough,
$(Y^N,Z^N,M^N) = (\tilde{Y}^N,\tilde{Z}^N,\tilde{M}^N)$, and
we may apply Theorem~\ref{thmBriand} to conclude that
%
%e5.4 #&#
\begin{equation}
\label{oneconv} \sup_t \biggl(\bigl|Y^{N}_t-
\tilde{Y}_t\bigr|+ \biggl|\int_0^t
Z^N_s \,\mathrm{d}W^N_s-\int
_0^t \tilde{Z}_s
\,\mathrm{d}W_s\biggr|+\bigl|M^{N}_t\bigr| \biggr) \stackrel{(N
\rightarrow\infty)} {\to} 0\qquad \mbox{in } L^2,
\end{equation}
and
%
%e5.5 #&#
\begin{eqnarray}
\label{secondconv}
&&\sup_t \Biggl(\sum_{k=1}^d
\biggl\llvert \int_0^t Z^{N,k}_s
\,\mathrm{d} \bigl\anglel W^N\bigr\angler_s - \int
_0^t \tilde{Z}^k_s \,\mathrm{d}s
\biggr\rrvert^2 + \biggl\llvert \int_0^t
\bigl|Z^N_s\bigr|^2 \,\mathrm{d} \bigl\anglel W^N
\bigr\angler_s - \int_0^t |
\tilde{Z}_s|^2 \,\mathrm{d}s\biggr\rrvert \Biggr)\nonumber\\[-8pt]\\[-8pt]
&&\quad \stackrel{(N
\rightarrow\infty)} {\to} 0\nonumber
\end{eqnarray}
in $L^1$. It follows from \eqref{secondconv} that $|\tilde{Z}| \le R$
$\nu\otimes\p$-almost everywhere.
So $(\tilde{Y}, \tilde{Z})$ is also a solution of the original BSDE
corresponding to $(f,\xi)$, and it follows from
Theorem~\ref{Morlais} that it is equal to $(Y,Z)$. This completes the proof.
\end{pf}

Another result that we need below is the following proposition.
%
%pr5.7 #&#
\begin{proposition}[(Briand and Hu~\cite{6})]\label{propBH}
Let $(\xi_m)_{m\in\mathbb{N}}$ be a sequence of $\F_T$-measurable
random variables such that $\sup_m \Vert\xi_m\Vert_\infty< \infty$ and
$\xi_m \to\xi$ almost surely. Furthermore assume that $f$ satisfies
\eqref{f1}.
Let $(Y_m,Z_m)$ and $(Y,Z)$ be solutions of the BSDEs
corresponding to $(f,\xi_m)$ and $(f,\xi)$, respectively, such that
$Y_m$ and $Y$ are bounded. If $Y_m$ is increasing (or decreasing) in
$m$, then
\[
\sup_t |Y_{m,t} - Y_t| \to0\qquad \mbox{a.s.}\quad
\mbox{and}\quad {\mathbb{E}} \biggl[\int_0^T
|Z_{m,s} - Z_s|^2 \,\mathrm{d}s \biggr] \to0 \qquad\mbox {for }
m \to\infty.
\]
\end{proposition}

%re5.8 #&#
\begin{remark}\label{variationHu}
Note that if $f$ satisfies \eqref{f1}--\eqref{f3}, then Proposition
\ref{propBH} holds
without the assumption that $Y_m$ is increasing or decreasing in $m$.
Indeed, by Theorem~\ref{Morlais} one has $Y(\xi_1)\geq Y(\xi_2)$ for
$\xi_1\geq\xi_2$
(where $Y(\xi)$ denotes the solution of the BSDE with driver $f$ and
terminal condition $\xi$). Define $\hat{\xi}_m=\sup_{n\geq m} \xi_n$ and
$\tilde{\xi}_m=\inf_{n\geq m} \xi_n$. Then one obtains from
Proposition~\ref{propBH}
that $\sup_t |Y_t(\hat{\xi}_m) - Y_t(\xi)| \to0$ and
$\sup_t |Y_t(\tilde{\xi}_m) - Y_t(\xi)| \to0$ a.s., and therefore also
$\sup_t |Y_t(\xi_m) - Y_t(\xi)| \to0$ a.s. The convergence of
$Z(\xi_m)$ to $Z(\xi)$
now follows exactly as in the proof of Proposition 2.4 in Kobylanski~\cite{21}.
\end{remark}

The next theorem shows that for any continuous-time terminal
condition there exists a sequence of discrete-time terminal conditions
such that the corresponding solutions of the BS$\Delta$Es
converge to their counterparts in continuous time.

%th5.9 #&#
\begin{theorem}
\label{thmone}
Assume there exists a $q \in[1,2)$ such that \emph{(W1)} and \emph{(f1)--(f5)} are satisfied.
Then for every $\xi\in L^{\infty}(\mathcal{F}_T)$, there exist
$\mathcal{F}^N_T$-measurable $\tilde{\xi}^N$ bounded by $\llVert
\xi\rrVert_{\infty}$ such that for $N$ large enough, the $N$th
BS$\Delta$E
with terminal condition $\tilde{\xi}^N$ has a unique solution
$(\tilde{Y}^N, \tilde{Z}^N, \tilde{M}^N)$ and
%
%e5.6 #&#
\begin{equation}
\label{Th1a} \sup_t \biggl(\bigl|\tilde{Y}^N_t-Y_t\bigr|
+ \biggl|\int_0^t \tilde{Z}^N_s
\,\mathrm{d}W^N_s - \int_0^t
Z_s \,\mathrm{d}W_s\biggr|+ \bigl|\tilde{M}^N_t\bigr|
\biggr) \stackrel{(N \rightarrow\infty)} {\to} 0\qquad \mbox{in } L^2
\end{equation}
as well as
%
%e5.7 #&#
\begin{eqnarray}
\label{Th1b} &&\sup_t \Biggl(\sum_{k=1}^d
\biggl\llvert \int_0^t \tilde{Z}^{N,k}_s
\,\mathrm{d} \bigl\anglel W^N\bigr\angler_s - \int
_0^t Z^k_s \,\mathrm{d}s\biggr
\rrvert^2 + \biggl\llvert \int_0^t
\bigl|\tilde{Z}^N_s\bigr|^2 \,\mathrm{d} \bigl\anglel
W^N\bigr\angler_s - \int_0^t
|Z_s|^2 \,\mathrm{d}s\biggr\rrvert \Biggr)\nonumber\\[-8pt]\\[-8pt] &&\quad\stackrel{(N\rightarrow
\infty)} {\to} 0\qquad \mbox{in } L^1,\nonumber
\end{eqnarray}
where $(Y,Z)$ is the unique solution of the BSDE \eqref{bsde} with
bounded $Y$.
Moreover, if $\xi= \varphi(W_{s_1},\ldots,W_{s_n})$ and $\xi^N =
\varphi(W^N_{s_1},\ldots,W^N_{s_n})$
for a bounded, uniformly continuous function $\varphi\dvtx  \mathbb{R}^{d
\times n} \to\mathbb{R}$, then
\[
\sup_t \bigl|Y^N_t-Y_t\bigr| \to0\qquad \mbox{in
} L^2 \mbox{ as } N \rightarrow \infty,
\]
where $(Y^N,Z^N,M^N)$ solves the $N$th BS$\Delta$E with terminal
condition $\xi^N$.
\end{theorem}

\begin{pf}
Given a random variable $\xi\in L^{\infty}(\mathcal{F}_T)$, there
exists a sequence $n_m$, $m \in\mathbb{N}$,
of positive integers together with times $0\le s^m_1< \cdots<s^m_{n_m}
\le T$ and Lipschitz-continuous
functions $\varphi_m \dvtx  \mathbb{R}^{d \times n_m} \to\mathbb{R}$
bounded by $\llVert \xi\rrVert_{\infty}$
such that the random variables $\xi_m := \varphi_m(W_{s_1},\ldots
,W_{s_{n_m}})$
converge to $\xi$ almost surely. It follows from (f1)--(f5) that the
driver $f$ satisfies
\eqref{f1}--\eqref{f3}. So one obtains from Theorem~\ref{Morlais}
that there exist unique solutions $(Y, Z)$
and $(Y_m,Z_m)$ to the \mbox{BSDEs} corresponding to $(f,\xi)$ and $(f,\xi_m)$, respectively, such that
$Y$ and $Y_m$ are bounded. Since for fixed $m$, $\varphi_m$ is bounded
and Lipschitz-continuous, one can apply
Proposition~\ref{PropLipschconv} and choose $N_m \in\mathbb{N}$
increasing in $m$ such that
for all $N \ge N_m$, one has
\begin{eqnarray*}
&& \mathbb{E} \Biggl[\sup_t \Biggl(\bigl|Y^N_{m,t}
- Y_{m,t}\bigr|^2 + \biggl|\int_0^t
Z^N_{m,s} \,\mathrm{d}W^N_s-\int
_0^t Z_{m,s} \,\mathrm{d}W_s
\biggr|^2 + \bigl|M^N_{m,t}\bigr|^2
\\
&&\hspace*{32pt}{} + \sum_{k=1}^d \biggl| \int
_0^t Z^{N,k}_{m,s} \,\mathrm{d} \bigl
\anglel W^N\bigr\angler_s-\int_0^t
Z^k_{m,s} \,\mathrm{d}s \biggr|^2
\\
&&\hspace*{32pt} {}+ \biggl|\int_0^t \bigl|Z^N_{m,s}\bigr|^2\,
\mathrm{d} \bigl\anglel W^N\bigr\angler_s-\int
_0^t |Z_{m,s}|^2 \,\mathrm{d}s \biggr|
\Biggr) \Biggr] \le\frac{1}{m},
\end{eqnarray*}
where $(Y^N_m,Z^N_m,M^N_m)$ is the unique solution to the $N$th
BS$\Delta$E with
driver $f^N$ and terminal condition $\xi^N_m := \varphi_m(W^N_{s_1},\ldots,W^N_{s_{n_m}})$.
Now set $\tilde{\xi}^N := \xi^N_{m_N}$ and
$(\tilde{Y}^N,\tilde{Z}^N,\tilde{M}^N) :=
(Y^N_{m_N},Z^N_{m_N},M^N_{m_N})$, where for given $N$,
$m_N$ is the largest $m$ satisfying $N_m \le N$. Then $\lim_{N \to
\infty} m_N = \infty$, and therefore,
\begin{eqnarray*}
&& \mathbb{E} \Biggl[\sup_t \Biggl( \bigl|\tilde{Y}_t^N
- Y_{m_N,t}\bigr|^2 +\biggl |\int_0^t
\tilde{Z}^N_s \,\mathrm{d}W^N_s-\int
_0^t Z_{m_N,s} \,\mathrm{d}W_s
\biggr|^2 +\bigl|\tilde{M}^N_t\bigr|^2
\\
&& \hspace*{32pt}{}+ \sum_{k=1}^d \biggl|\int
_0^t \tilde{Z}^{N,k}_s\, \mathrm{d}
\bigl\anglel W^N\bigr\angler_s -\int_0^t
Z^k_{m_N,s} \,\mathrm{d}s\biggr |^2
\\
&&\hspace*{32pt}{} + \biggl|\int_0^t \bigl|\tilde{Z}^N_s\bigr|^2
\,\mathrm{d}\bigl\anglel W^N\bigr\angler_s-\int
_0^t |Z_{m_N,s}|^2 \,\mathrm{d}s \biggr|
\Biggr) \Biggr] \stackrel{(N \rightarrow\infty)} {\rightarrow } 0.
\end{eqnarray*}
In particular,
\[
\sup_{t\in[0,T]} \bigl|\tilde{M}^N_t\bigr| \stackrel{(N
\rightarrow\infty )} {\to} 0 \qquad\mbox{in }L^2.
\]
Moreover, it follows from Proposition~\ref{propBH} and Remark \ref
{variationHu} that
\[
\sup_t |Y_{m_N,t} - Y_t| \to0\qquad \mbox{a.s.}\quad
\mbox{and}\quad {\mathbb {E}} \biggl[\int_0^T
|Z_{m_N,s} - Z_s|^2 \,\mathrm{d}s \biggr] \to0.
\]
This implies \eqref{Th1a}--\eqref{Th1b}.

If
\[
\xi=\varphi(W_{s_1},\ldots,W_{s_n}) \quad\mbox{and}\quad
\xi^N=\varphi\bigl(W^N_{s_1},
\ldots,W^N_{s_n}\bigr)
\]
for a bounded, uniformly continuous function $\varphi\dvtx  \mathbb{R}^{d
\times n} \to\mathbb{R}$,
there exist Lipschitz-continuous functions
$\varphi_m \dvtx  \mathbb{R}^{d \times N} \to\mathbb{R}$ bounded by
$\llVert \varphi\rrVert_{\infty}$
such that $\sup_{x \in\mathbb{R}^{d \times n}} |\varphi_m(x) -
\varphi(x)| \le1/m$.
Choose $m_N$ as in the first part of the proof and set
\[
\tilde{\xi}^N := \varphi_{m_N}\bigl(W^N_{s_1},
\ldots,W^N_{s_n}\bigr).
\]
One then obtains as above that
\[
\sup_t \bigl|\tilde{Y}^N_t - Y_t\bigr|
\to0 \qquad\mbox{in } L^2 \mbox{ as } N \to \infty.
\]
By Lemma~\ref{infterm2}, there exists an $N_0 \in\mathbb{N}$
and a constant $D \in\mathbb{R}_+$ such that for $N \ge N_0$,
\[
\sup_t \bigl|Y^N_t- \tilde{Y}^N_t\bigr|
\le D\bigl\Vert\xi^N-\tilde{\xi}^N\bigr\Vert_\infty.
\]
Hence,
\[
\sup_t \bigl|Y^N_t - \tilde{Y}^N_t\bigr|
\to0 \qquad\mbox{in } L^2 \mbox{ for } N \to\infty,
\]
and one can conclude that
\[
\sup_t \bigl|Y^N_t - Y_t\bigr| \to0\qquad
\mbox{in } L^2 \mbox{ for } N \to\infty.
\]
\upqed\end{pf}

In the following corollary, we denote by $C^d[0,T]$ the set of
all continuous functions from $[0,T]$ to $\mathbb{R}^d$ and assume that
the driver $f$ is of the form
%
%e5.8 #&#
\begin{equation}
\label{ftildef} f(t,y,z) = \tilde{f}(t,W,y,z)
\end{equation}
for a measurable function $\tilde{f} \dvtx  [0,T] \times C^d[0,T] \times
\mathbb{R}
\times\mathbb{R}^d \to\mathbb{R}$ that is left-continuous in $t$
and for which there exists a $q \in[1,2)$ such that
conditions \eqref{e9}--\eqref{f4} are satisfied:
%
%e5.9 #&#
%e5.10 #&#
\begin{eqnarray}\label{e9} \bigl|\tilde{f}(t,w,y,z)\bigr|  &\le& K\bigl(1+|y|+|z|^q\bigr)
\qquad\mbox{for all } t,w,y,z,
\\
\label{e10}  \bigl|\tilde{f}(t,w,y_1,z)-\tilde{f}(t,w,y_2,z)\bigr|
&\le& K|y_1-y_2|\qquad \mbox{for all } t,w,y_1,
y_2,z.
\end{eqnarray}
For every $a \in\mathbb{R}_+$ there exists $b \in\mathbb{R}_+$ such that
%
%e5.11 #&#
\begin{equation}\label{e11} \bigl|\tilde{f}(t,w,y,z_1)- \tilde{f}(t,w,y,z_2)\bigr|
\le b\bigl(1+\bigl(|z_1| \vee |z_2|\bigr)^{q/2}\bigr)
|z_1-z_2|
\end{equation}
for all $t \in[0,T]$, $y \in[-a,a]$ and $z_1,z_2 \in\mathbb{R}^d$.

There exists a constant $L \in\mathbb{R}_+$ such that
%
%e5.12 #&#
\begin{equation}
\label{f4} \bigl|\tilde{f}(t,w_1,y,z) - \tilde{f}(t,w_2,y,z)\bigr|
\le L \sup_{s \le t} \bigl|w_1(s) - w_2(s)\bigr|\qquad \mbox{for all }
t,w_1,w_2,y,z.
\end{equation}
We also assume that the discrete-time drivers $f^N$ are of the form
%
%e5.13 #&#
\begin{equation}
\label{fNtildef} f^N\bigl(t,W^N,y,z\bigr) = \tilde{f}
\bigl(t^N_{i+1},W^{N,c},y,z\bigr) \qquad\mbox{for }
t^N_i < t \le t^N_{i+1},
\end{equation}
where $W^{N,c}$ is the following continuous approximation of $W^N$:
Set $h^N = \sup_i |t^N_i - t^N_{i-1}|$ and
\[
W^{N,c}_t = \cases{ 0, &\quad for $t \le h^N$,
\cr
W^N_{t^N_{i-1}} + \dfrac{t - (t^N_{i-1} + h^N)}{t^N_i - t^N_{i-1}} \bigl(W^N_{t^N_i}
- W^N_{t^N_{i-1}}\bigr), &\quad for $t^N_{i-1} +
h^N \le t \le t^N_i + h^N$. }
\]
Note that $W^{N,c}$ is adapted to the filtration $(\mathcal{F}^N_t)$
and $f^N(t^N_{i+1},W^N,y,z)$ only depends on
$W^N_{t^N_1}, \ldots, W^N_{t^N_i}$.

%co5.10 #&#
\begin{corollary} \label{coro1}
Assume the $W^N$ fulfill \emph{(C1)}, \emph{(C2)} and \emph{(W1)} for some
$q \in[1,2)$,
but instead of \emph{(C3)} they converge to $W$ in distribution and satisfy
$\sup_N {\mathbb{E}} [\sup_t |W^N_t|^{2+\varepsilon} ] <
\infty$ for some
$\varepsilon> 0$.
Furthermore, suppose $f$ and $f^N$ are of the form \eqref{ftildef} and
\eqref{fNtildef},
respectively. Then for every $\xi\in L^{\infty}(\mathcal{F}_T)$,
there exists a sequence of
$\mathcal{F}^N_T$-measurable random variables $\tilde{\xi}^N$
bounded by $\llVert \xi\rrVert_{\infty}$ such that for $N$ large enough,
the $N$th BS$\Delta$E with terminal condition $\tilde{\xi}^N$ has a
unique solution
$(\tilde{Y}^N, \tilde{Z}^N, \tilde{M}^N)$ and
\[
\sup_t \bigl|\tilde{Y}^N_t - Y_t\bigr|
\to0 \qquad\mbox{in distribution for } N \to \infty,
\]
where $(Y,Z)$ is the unique solution of the BSDE \eqref{bsde} with
bounded $Y$.
In the special case, where $\xi= \varphi(W_{s_1},\ldots,W_{s_{n}})$
for a uniformly continuous function $\varphi\dvtx  \mathbb{R}^{d \times n}
\to\mathbb{R}$,
one can choose $\tilde{\xi}^N = \varphi(W^N_{s_1},\ldots,W^N_{s_{n}})$.
\end{corollary}

\begin{pf}
It can be shown as in Example~\ref{ExBernoulli} that there exists
a probability space $(\tilde{\Omega}, \tilde\mathcal{F}, \tilde{\p})$
supporting a $d$-dimensional Brownian motion $\tilde{W}$ and
random walks $\tilde{W}^N$ with the same distributions as
$W^N$ such that ${\mathbb{E}} [\sup_t |\tilde{W}^N_t - \tilde
{W}_t|^2 ]\to0$
for $N \to\infty$. Then
\[
\sup_t \bigl|\tilde{f}\bigl(t,\tilde{W}^{N,c},y,z\bigr) -
\tilde{f}(t,\tilde {W},y,z)\bigr| \to0 \qquad\mbox{in } L^2 \mbox{ for } N \to
\infty
\]
and it follows from Theorem~\ref{thmone} that
for every $\tilde{\xi} \in L^{\infty}(\tilde\mathcal{F}_T)$ one can
choose $\mathcal{F}^N_T$-measurable terminal conditions $\tilde{\xi
}^N$ bounded by $\llVert \xi\rrVert_{\infty}$
such that the corresponding solutions satisfy
$\sup_t |\tilde{Y}^N_t - \tilde{Y}_t| \to0$ in $L^2$ as $N \to
\infty$.
Furthermore, if $\tilde{\xi}$ is of the form
$\tilde{\xi} = \varphi(\tilde{W}_{s_1}, \ldots, \tilde{W}_{s_n})$
for a
uniformly continuous function $\varphi\dvtx  \mathbb{R}^{d \times n} \to
\mathbb{R}$,
one can choose $\tilde{\xi}^N = \varphi(\tilde{W}^N_{s_1}, \ldots,
\tilde{W}^N_{s_n})$.
This proves the corollary.
\end{pf}

%ex5.11 #&#
\begin{example}
In the setting of Corollary~\ref{coro1}, let $\xi= \varphi(W_T)$ and
$\xi^N = \varphi(W^N_T)$ for
\[
\varphi(x) = \cases{ \sqrt{x}\wedge1, &\quad if $x \geq0$,
\cr
-\sqrt{-x}\vee-1, &\quad if $x
< 0$. }
\]
Then for every function $\tilde{f}$ satisfying \eqref{e9}--\eqref
{f4} the corresponding solutions $Y^N$ converge to $Y$ in
distribution. Let us illustrate this result for the example
\[
\tilde{f}(t,w,y,z) = K_1 y+ K_2 |z|^{3/2}.
\]
Let $T =1$ and $W^N$ be the Bernoulli random walks from Example \ref
{ExBernoulli}. Then the
discrete equations can numerically be solved using formulas \eqref
{Yformula}--\eqref{Zformula}.

%f1 #&#
\begin{figure}
\centering
\begin{tabular}{@{}cc@{}}

\includegraphics{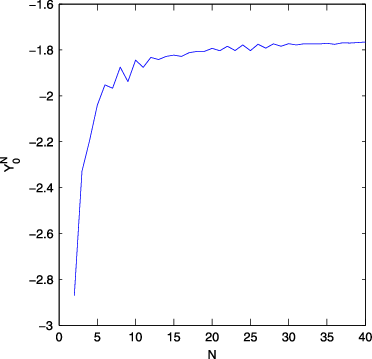}
&\includegraphics{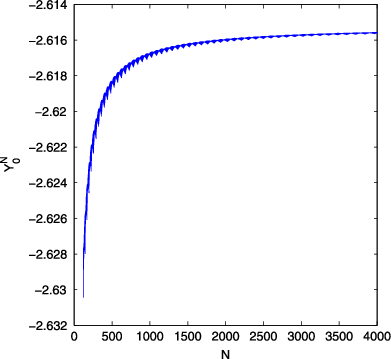}\\
\footnotesize{(a)}&\footnotesize{(b)}
\end{tabular}
\caption{(a) $Y^N_0$ corresponding to $K_1=1$ and $K_2=1$. (b) $Y^N_0$ corresponding to $K_1=1$ and $K_2=5$.}\label{figthree}
\end{figure}
Figure~\ref{figthree}(a) and (b) show the convergence of $Y^N_0$ for different
values of $K_1$ and $K_2$. It can be seen that for
$(K_1,K_2) = (1,1)$, $Y^N_0$ converges rather fast. Already for $N=20$,
it is close to the limit value. On the other hand, for
$(K_1, K_2) = (1,5)$, the convergence is much slower.
\end{example}

%s6 #&#
\section{Convergence results for convex drivers}\label{secvex}

In this section, we consider BS$\Delta$Es with drivers that are
convex in $z$ and use convex duality to derive stronger convergence
results than in
Section~\ref{secconvergence}. For the case where
$f$ does not depend on $y$ it has been shown in Barrieu and El Karoui
\cite{2}, Delbaen \textit{et~al.}~\cite{14} and Delbaen \textit{et al.}~\cite{15} that
BSDEs with convex drivers admit a convex dual representation.
Here, we establish convex dual representations for solutions of
BS$\Delta$Es and use them to show convergence.
We need the following stronger version of condition (W1) on the
approximating processes $W^N$:
\begin{enumerate}[(W2)]
\item[(W2)] ${\mathbb{E}} [\Delta W^{N,k}_{t^N_i} \Delta
W^{N,l}_{t^N_i} ] = 0$
for all $N \in\mathbb{N}$, $i=1,\ldots,i_N$, $k \neq l$ and
\[
\sup_{N,i,k} \frac{\llVert \Delta W^{N,k}_{t^N_i} \rrVert_{\infty
}}{\sqrt {\Delta\anglel W^{N}\angler_{t^N_i}}} < \infty.
\]
\end{enumerate}
Note that this implies (W1) for all $q \in[1,2)$.
In the following, we assume that the drivers $f^N$ are convex in $z$
and define
\[
g^N(t,y,\mu) := \mathop{\ess\sup}_z \bigl\{\mu
z-f^N(t,y,z) \bigr\},\qquad \mu \in\mathbb{R}^d.
\]
Let $\mu^N$ be an $\mathbb{R}^d$-valued $(\F^N_t)$-adapted process
that is constant on the intervals $(t^N_{i-1}, t^N_i]$ and satisfies
%
%e6.1 #&#
\begin{equation}
\label{dens} \mu^N_{t^N_i} \Delta W^N_{t^N_i}
> -1\qquad \mbox{for all } i.
\end{equation}
Then
%
%e6.2 #&#
\begin{equation}
\label{dens2} \frac{\mathrm{d}\p^{\mu^N}}{\mathrm{d}\p} = \prod_{i=1}^{i_N}
\bigl(1 + \mu^N_{t^N_i} \Delta W^{N}_{t^N_i}
\bigr)
\end{equation}
defines a probability measure $\p^{\mu^N}$ equivalent to $\p$ under
which the processes
\[
W^{N, \mu^N,k}_{t^N_i} =W^{N,k}_{t^N_i} - \sum
_{j=1}^i \mu^k_{t^N_j}\Delta
\bigl\anglel W^{N}\bigr\angler_{t^N_j},\qquad k = 1, \ldots , d,
\]
are martingales. The following proposition gives an implicit dual
representation of solutions of BS$\Delta$Es. Its proof can be found in
the \hyperref[app]{Appendix}.

%pr6.1 #&#
\begin{proposition} \label{gN}
Assume \emph{(W2)} and let $C,K,L \in\mathbb{R}_+$,
$q \in[1,2)$ be constants such that all terminal conditions $\xi^N$
and drivers $f^N$ fulfill the following conditions:
\begin{enumerate}[(iii)]
\item[(i)]
$\llVert \xi^N \rrVert_{\infty} \le C$;
\item[(ii)]
$f^N$ is convex in $z$;
\item[(iii)]
$|f^N(t,y,z)| \le K(1+|y|+|z|^q)$ for all $(t,y,z) \in[0,T] \times
\mathbb{R}^{d+1}$;
\item[(iv)]
$|f^N(t,y_1,z) - f^N(t,y_2,z)| \le L |y_1-y_2|$
for all $(t,y,z) \in[0,T] \times\mathbb{R}^{d+1}$;
\item[(v)]
$|f^N(t,y,z_1)-f^N(t,y,z_2)| \le L(1+(|z_1| \vee|z_2|)^{q/2})
|z_1-z_2|$
for all $(t,y,z_1,z_2) \in[0,T] \times\mathbb{R}^{2d +1}$ such that
$|y| \le(C+1) \exp(K T)$.
\end{enumerate}
Then there exists $N_0 \in\mathbb{N}$ such that for every $N \geq N_0$,
the $N$th BS$\Delta$E has a unique solution $(Y^N,Z^N,M^N)$ and $Y^N$ can
be represented as
%
%e6.3 #&#
\begin{equation}
\label{dual} Y^N_{t^N_i} = \mathop{\ess\sup}_{\mu^N} \mathbb{E}^{\mu^N} \Biggl[\xi^N- \sum
_{j = i+1}^{i_N} g^N
\bigl(t^N_{j},Y^N_{t^N_{j-1}},
\mu^N_{t^N_{j}}\bigr)\Delta\bigl\anglel W^N\bigr
\angler_{t^N_{j}}\Big|\F^N_{t^N_i} \Biggr],
\end{equation}
where the essential supremum is taken over all $\mathbb{R}^d$-valued
$(\mathcal{F}^N_t)$-adapted processes $\mu^N$ that are constant on
the intervals $(t^N_{i-1}, t^N_i]$ and satisfy \eqref{dens}.
Moreover, there exists a constant $R \in\mathbb{R}_+$
such that for each $N \ge N_0$, \eqref{dual} admits a maximizer $\hat
{\mu}^N$
satisfying
%
%e6.4 #&#
\begin{equation}
\label{Repsilon} \mathbb{E}^{\hat{\mu}^N} \Biggl[\sum
_{j = i +1}^{i_N} |\hat{\mu }_{t^N_j}|^2
\Delta\bigl\anglel W^N\bigr\angler_{t^N_j}\Big |
\F^N_{t^N_i} \Biggr] \le R \qquad\mbox{for all } i \le
i_N-1.
\end{equation}
\end{proposition}

We are now ready to prove our convergence result for convex drivers. It
states that for any sequence of bounded discrete-time terminal
conditions converging to $\xi$ and every sequence of discrete-time
drivers converging to $f$ the discrete-time solutions $Y^N$ converge to the
continuous-time solution $Y$.

%th6.2 #&#
\begin{theorem} \label{thmtwo}
Assume \emph{(W2)}, the $f^N(t,y,z)$ are convex in $z$ and one has
$\sup_N \llVert \xi^N \rrVert_{\infty} < \infty$ as well as $\xi^N \to\xi$
in $L^2$.
Moreover, suppose the $f^N$ satisfy \emph{(f1)--(f5)}. Then for $N$
large enough, the $N$th BS$\Delta$E
has a unique solution $(Y^N,Z^N,M^N)$ and
\[
\sup_t \bigl|Y^N_t-Y_t\bigr| \rightarrow0\qquad
\mbox{in $L^2$ for } N \to\infty,
\]
where $(Y,Z)$ is the unique solution of the BSDE \eqref{bsde} with
bounded $Y$.
\end{theorem}

\begin{pf}
By Theorem~\ref{thmone}, there exist $\mathcal{F}^N_T$-measurable
terminal conditions
$\tilde{\xi}^N$ bounded by $C := \sup_N \llVert \xi^N \rrVert_{\infty}$ such
that the corresponding solutions satisfy
\[
\sup_t \bigl|\tilde{Y}^N_t - Y_t\bigr|
\to0 \qquad\mbox{in } L^2.
\]
Choose $b \in\mathbb{R}_+$ such that condition (f3) holds for
$a = (C+1) \exp(KT)$. Then the conditions of Theorem~\ref{thmcomp} and
Proposition~\ref{gN} are satisfied with $L = K \vee b$. Hence, there exists
$N_0 \in\mathbb{N}$ such that for all $N \ge N_0$,
$\sup_t|Y^N_t|$ and $\sup_t|\tilde{Y}^N_t|$ are bounded by $(C+1)
\exp(KT)$
and
\begin{eqnarray*}
Y^N_{t^N_i} &=& \mathop{\ess\sup}_\mu{
\mathbb{E}}^\mu \Biggl[\xi^N - \sum
_{j=i+1}^{i_N} g^N\bigl(t^N_{j},Y^N_{t^N_{j-1}},
\mu_{t^N_{j}} \bigr) \Delta\bigl\anglel W^N\bigr
\angler_{t^N_{j}} \Big| \mathcal{F}^N_{t^N_i} \Biggr]
\\
&=& \mathbb{E}^{\hat{\mu}^N} \Biggl[\xi^N - \sum
_{j=i+1}^{i_N} g^N\bigl(t^N_{j},Y^N_{t^N_{j-1}},
\hat{\mu}^N_{t^N_{j}} \bigr) \Delta\bigl\anglel W^N
\bigr\angler_{t^N_{j}} \Big| \mathcal{F}^N_{t^N_i} \Biggr]
\end{eqnarray*}
as well as
\begin{eqnarray*}
\tilde{Y}^N_{t^N_i} &=& \mathop{\ess\sup}_\mu{
\mathbb{E}}^\mu \Biggl[\tilde{\xi}^N - \sum
_{j=i+1}^{i_N} g^N\bigl(t^N_{j},
\tilde {Y}^N_{t^N_{j-1}},\mu_{t^N_{j}} \bigr) \Delta\bigl
\anglel W^N\bigr\angler_{t^N_{j}} \Big| \mathcal{F}^N_{t^N_i}
\Biggr]
\\
&=& \mathbb{E}^{\tilde{\mu}^N} \Biggl[\tilde{\xi}^N - \sum
_{j=i+1}^{i_N} g^N\bigl(t^N_{j},
\tilde{Y}^N_{t^N_{j-1}},\tilde{\mu }^N_{t^N_{j}}
\bigr) \Delta\bigl\anglel W^N\bigr\angler_{t^N_{j}} \Big| \mathcal
{F}^N_{t^N_i} \Biggr].
\end{eqnarray*}
If we can show
\[
\sup_t \bigl| \tilde{Y}^N_t-Y^N_t\bigr|
\to0 \qquad\mbox{in } L^2,
\]
we get
\[
\sup_t \bigl|Y^N_t-Y_t\bigr| \to0 \qquad\mbox{in
} L^2,
\]
and the theorem is proved. As the supremum of $K$-Lipschitz functions, $g^N$
is again $K$-Lipschitz in $y$. Hence, since
$|\max \{a_1, a_2 \} - \max \{b_1, b_2 \}| \le
\max \{|a_1-b_1|,|a_2-b_2| \}$
for $a_1,a_2,b_1,b_2 \in\mathbb{R}$, and
\begin{eqnarray*}
Y^N_{t^N_i} &=& \max_{\mu\in \{\hat{\mu}^N,\tilde{\mu
}^N \}} {\mathbb{E}}^\mu
\Biggl[\xi^N -\sum_{j=i+1}^{i_N}
g^N\bigl(t^N_{j},Y^N_{t^N_{j-1}},
\mu_{t^N_{j}}\bigr) \Delta\bigl\anglel W^N\bigr
\angler_{t^N_{j}} \Big\mid\mathcal{F}^N_{t^N_i} \Biggr],
\\
\tilde{Y}^N_{t^N_i} &=& \max_{\mu\in\{\hat{\mu}^N,\tilde{\mu
}^N\}} {
\mathbb{E}}^\mu \Biggl[\tilde{\xi}^N -\sum
_{j=i+1}^{i_N} g^N\bigl(t^N_{j},
\tilde{Y}^N_{t^N_{j-1}},\mu_{t^N_{j}} \bigr) \Delta\bigl
\anglel W^N\bigr\angler_{t^N_{j}} \Big| \mathcal{F}^N_{t^N_i}
\Biggr],
\end{eqnarray*}
one obtains
\begin{eqnarray*}
\bigl|\tilde{Y}^N_{t^N_i}-Y^N_{t^N_i}\bigr| &\le&
\max_{\mu\in \{\hat{\mu}^N,\tilde{\mu}^N \}} {\mathbb{E}}^\mu \Biggl[\bigl|\tilde{\xi}^N
-\xi^N\bigr| + K\sum_{j=i+1}^{i_N} \bigl|
\tilde{Y}_{t^N_{j-1}}^N - Y^N_{t^N_{j-1}}\bigr| \Delta
\bigl\anglel W^N\bigr\angler_{t^N_{j}} \Big\mid\mathcal{F}^N_{t^N_i}
\Biggr]
\\
&\le&\mathbb{E}^{\hat{\mu}^{N}} \Biggl[\bigl|\tilde{\xi}^N -
\xi^N\bigr| +K \sum_{j=i+1}^{i_N}\bigl |
\tilde{Y}_{t^N_{j-1}}^N -Y^N_{t^N_{j-1}}\bigr|\Delta
\bigl\anglel W^N\bigr\angler_{t^N_{j}} \Big\mid\mathcal
{F}^N_{t^N_i} \Biggr]
\\
&& {}+\mathbb{E}^{\tilde{\mu}^{N}} \Biggl[\bigl|\tilde{\xi}^N -
\xi^N\bigr| +K\sum_{j=i+1}^{i_N}\bigl |
\tilde{Y}_{t^N_{j-1}}^N-Y^N_{t^N_{j-1}}\bigr |\Delta
\bigl\anglel W^N\bigr\angler_{t^N_{j}} \Big\mid\mathcal{F}^N_{t^N_i}
\Biggr].
\end{eqnarray*}
From Proposition~\ref{gN}, we know that there exists a constant $R \in
\mathbb{R}_+$ such that
\[
\mathbb{E}^{\hat{\mu}^{N}} \Biggl[\sum_{j=i+1}^{i_N}
\bigl|\hat{\mu}^{N}_{t^N_{j}}\bigr|^2 \Delta\bigl\anglel
W^N\bigr\angler_{t^N_{j}}\Big |\F^N_{t^N_i}
\Biggr] \le R\qquad \mbox{for all } N \ge N_0 \mbox{ and } i=0,
\ldots,i_N-1.
\]
Consequently, we obtain from Lemma~\ref{2mubound} in the
\hyperref[app]{Appendix} that there exists a constant $\tilde{R}$ such that
\[
{\mathbb{E}} \Biggl[\varphi \Biggl(\prod_{j = i+1}^{i_N}
\bigl(1 + \hat{\mu }^N_{t^N_j} \Delta W^N_{t^N_j}
\bigr) \Biggr) \Big| \F^N_{t^N_i} \Biggr] \le \tilde{R}\qquad \mbox{for
all } N \ge N_0 \mbox{ and } i=0,\ldots,i_N-1,
\]
where $\varphi(x) = x \log(x) \vee1$. Fix $\varepsilon>0$ and set
$D= 2[C + (C+1) \exp(KT) K \sup_{N \ge N_0} \anglel W^N\angler_T]$.
Since $\varphi(x)/x \uparrow\infty$, there exists $B \in\mathbb{R}_+$
such that for all $x> B$,
\[
\frac{x}{\varphi(x)} \le\frac{\varepsilon}{\tilde{R} D}.
\]
Introduce the sets $E^N_{i+1} = \{\prod_{j = i+1}^{i_N}(1 + \hat{\mu
}^N_{t^N_j} \Delta W^N_{t^N_j}) > B\}$.
Then
%
%e6.5 #&#
\begin{eqnarray}\label{estimateE}
&& \sup_{N \ge N_0 , 0 \le i \le i_N-1} \mathbb{E} \Biggl[1_{E^N_{i+1}} \prod
_{j = i+1}^{i_N} \bigl(1 + \hat{\mu}^N_{t^N_j}
\Delta W^N_{t^N_j}\bigr) \Big|\mathcal {F}^N_{t^N_i}
\Biggr]
\nonumber
\\
&&\quad= \sup_{N \ge N_0 , 0 \le i \le i_N-1} {\mathbb{E}} \Biggl[1_{E^N_{i+1}}
\frac{\prod_{j = i+1}^{i_N}(1 + \hat{\mu}^N_{t^N_j}
\Delta W^N_{t^N_j})}{\varphi(\prod_{j = i+1}^{i_N} (1 + \hat{\mu
}^N_{t^N_j} \Delta W^N_{t^N_j}))}\\[-2pt]
&&\hspace*{92pt}{}\times \varphi\Biggl(\prod_{j = i+1}^{i_N}
\bigl(1 + \hat{\mu}^N_{t^N_j} \Delta W^N_{t^N_j}
\bigr)\Biggr) \Big|\mathcal {F}^N_{t^N_i} \Biggr]\nonumber
\\[-2pt]
 &&\quad\le \frac{\varepsilon}{\tilde{R}D} \sup_{N \ge N_0 , 0 \le i \le i_N-1} {\mathbb{E}}
\Biggl[\varphi\Biggl(\prod_{j = i+1}^{i_N}\bigl(1
+ \hat{\mu }^N_{t^N_j} \Delta W^N_{t^N_j}
\bigr)\Biggr)\Big|\mathcal{F}^N_{t^N_i} \Biggr] \le
\frac{\varepsilon}{D}.\nonumber
\end{eqnarray}

This yields for all $N \ge N_0$,
\begin{eqnarray*}
&& \mathbb{E}^{\hat{\mu}^{N}} \Biggl[\bigl|\tilde{\xi}^N -
\xi^N\bigr| + K \sum_{j=i+1}^{i_N} \bigl|
\tilde{Y}_{t^N_{j-1}}^N-Y^N_{t^N_{j-1}}\bigr| \Delta
\bigl\anglel W^N\bigr\angler_{t^N_{j}} \mid\mathcal{F}^N_{t^N_i}
\Biggr]
\\[-2pt]
&&\quad= \mathbb{E} \Biggl[\prod_{j = i+1}^{i_N}
\bigl(1 + \hat{\mu}^N_{t^N_j} \Delta W^N_{t^N_j}
\bigr) \Biggl(\bigl|\tilde{\xi}^N -\xi^N\bigr| + K \sum
_{j=i+1}^{i_N} \bigl|\tilde{Y}_{t^N_{j-1}}^N-Y^N_{t^N_{j-1}}\bigr|
\Delta\bigl\anglel W^N\bigr\angler_{t^N_{j}} \Biggr) \Big|
\mathcal{F}^N_{t^N_i} \Biggr]
\\[-2pt]
&&\quad= \mathbb{E} \Biggl[ 1_{E^N_{i+1}}\prod_{j = i+1}^{i_N}
\bigl(1 + \hat {\mu}^N_{t^N_j} \Delta W^N_{t^N_j}
\bigr) \Biggl(\bigl|\tilde{\xi}^N -\xi^N\bigr| + K \sum
_{j=i+1}^{i_N} \bigl| \tilde{Y}_{t^N_{j-1}}^N-Y^N_{t^N_{j-1}}\bigr|
\Delta\bigl\anglel W^N\bigr\angler_{t^N_{j}} \Biggr) \Big|
\mathcal{F}^N_{t^N_i} \Biggr]
\\[-2pt]
&&\qquad {}+ \mathbb{E} \Biggl[1_{E^{N,c}_{i+1}}\prod_{j = i+1}^{i_N}
\bigl(1 + \hat{\mu}^N_{t^N_j} \Delta W^N_{t^N_j}
\bigr) \Biggl(\bigl|\tilde{\xi}^N -\xi^N\bigr| + K \sum
_{j=i+1}^{i_N}\bigl |\tilde {Y}_{t^N_{j-1}}^N-Y^N_{t^N_{j-1}}\bigr|
\Delta\bigl\anglel W^N\bigr\angler_{t^N_{j}} \Biggr)\Big |
\mathcal{F}^N_{t^N_i} \Biggr]
\\[-2pt]
&&\quad\le D {\mathbb{E}} \Biggl[1_{E^N_{i+1}}\prod_{j = i+1}^{i_N}
\bigl(1 + \hat{\mu}^N_{t^N_j} \Delta W^N_{t^N_j}
\bigr) |\mathcal{F}^N_{t^N_i} \Biggr]
\\[-2pt]
&&\qquad {}+ \mathbb{E} \Biggl[ 1_{E^{N,c}_{t^N_{i+1}}}\prod_{j = i+1}^{i_N}
\bigl(1 + \hat{\mu}^N_{t^N_j} \Delta W^N_{t^N_j}
\bigr) \\[-2pt]&&\hspace*{92pt}{}\times\Biggl(\bigl|\tilde{\xi}^N -\xi^N\bigr| + K \sum
_{j=i+1}^{i_N} \bigl|\tilde{Y}_{t^N_{j-1}}^N-Y^N_{t^N_{j-1}}\bigr|
\Delta\bigl\anglel W^N\bigr\angler_{t^N_{j}} \Biggr)\Big |
\F^N_{t^N_i} \Biggr]
\\[-2pt]
&&\quad\le \varepsilon+ B {\mathbb{E}} \Biggl[\bigl|\tilde{\xi}^N -
\xi^N\bigr|+K\sum_{j=i+1}^{i_N}\bigl |\tilde
{Y}_{t^N_{j-1}}^N-Y^N_{t^N_{j-1}}\big|\Delta\bigl
\anglel W^N\bigr\angler_{t^N_{j}} \Big|\F^N_{t^N_i}
\Biggr].
\end{eqnarray*}
In the first inequality, we used that the random variables
$|\tilde{\xi}^N-\xi^N|+K\sum_{j=i+1}^{i_N}
|\tilde{Y}_{t^N_{j-1}}^N-Y^N_{t^N_{j-1}}| \Delta\anglel W^N\angler_{t^N_{j}}$
are uniformly bounded by $D$. In the second inequality,
we used \eqref{estimateE} and the definition of
the sets $E^N_{i+1}$. Using the same estimate for $\tilde{\mu}^{N}$
instead of $\hat{\mu}^{N}$ gives
%
%e6.6 #&#
\begin{equation}
\label{last2}\bigl |\tilde{Y}^N_{t^N_i}-Y^N_{t^N_i}\bigr|
\le2 \varepsilon+2B \mathbb{E} \Biggl[\bigl|\tilde{\xi}^N -
\xi^N\bigr|+K\sum_{j=i+1}^{i_N} \bigl|
\tilde{Y}_{t^N_{j-1}}^N -Y^N_{t^N_{j-1}}\bigr|\Delta
\bigl\anglel W^N\bigr\angler_{t^N_{j}}\Big |\F^N_{t^N_i}
\Biggr].\vadjust{\goodbreak}
\end{equation}
Taking expectations, one gets
\[
\mathbb{E} \bigl[\bigl|\tilde{Y}^N_{t^N_i}- Y^N_{t^N_{i}}\bigr|
\bigr] \le2\varepsilon+2B{\mathbb{E}} \bigl[\bigl|\tilde{\xi}^N -
\xi^N\bigr| \bigr] +K\sum_{j=i+1}^{i_N}
{\mathbb{E}} \bigl[\bigl|\tilde {Y}^N_{t^N_{j-1}}-Y^N_{t^N_{j-1}}\bigr|
\bigr]\Delta\bigl\anglel W^N\bigr\angler_{t^N_{j}}.
\]
Since $\varepsilon>0$ was arbitrary, one obtains from a discrete
version of Gronwall's lemma
(see Lemma~\ref{lemmagronwall} in the \hyperref[app]{Appendix}) that
\[
\sup_t {\mathbb{E}} \bigl[\bigl|\tilde{Y}^N_t
-Y^N_t\bigr| \bigr] \to0 \qquad\mbox {as } N \to\infty,
\]
and since $Y^N$ and $\tilde{Y}^N$ are both bounded by $(C+1) \exp
(KT)$, also
%
%e6.7 #&#
\begin{equation}
\label{l2} \sup_t {\mathbb{E}} \bigl[\bigl|\tilde{Y}^N_t
-Y^N_t\bigr|^2 \bigr] \to0 \qquad\mbox{as } N \to
\infty.
\end{equation}
It remains to show that $\sup_t$ can be taken inside of the
expectation in \eqref{l2}.
To do this, note that \eqref{last2} gives
\begin{eqnarray*}
\sup_{i}\bigl|\tilde{Y}_{t^N_{i}}^N-Y^N_{t^N_{i}}\bigr|
\le2\varepsilon +2B \Bigl(\sup_{i}{\mathbb{E}} \bigl[\bigl|\tilde{
\xi}^N -\xi^N\bigr||\F^N_{t^N_i} \bigr]+K
\sup_{i}A^N_{t^N_i} \Bigr)
\end{eqnarray*}
for the nonnegative martingale
\[
A^N_{t^N_i}= {\mathbb{E}} \Biggl[\sum
_{j=1}^{i_N} \bigl|\tilde {Y}_{t^N_{j-1}}^N-Y^N_{t^N_{j-1}}\bigr|
\Delta\bigl\anglel W^N\bigr\angler_{t^N_{j}} \Big|
\F^N_{t^N_i} \Biggr],\qquad i=0,\ldots,i_N.
\]
Since $\varepsilon$ was arbitrary, and
$\sup_{i}{\mathbb{E}} [|\tilde{\xi}^N -\xi^N| | \mathcal
{F}^N_{t^N_i} ]\stackrel{(N \rightarrow\infty)}{\to} 0$ in
$L^2$ by
Doob's maximal inequality, the only thing left to show is
$\sup_{i} A^N_{t^N_i} \stackrel{(N \rightarrow\infty)}{\to} 0$ in $L^2.$
Applying Doob's maximal inequality to $A^N$ yields
\begin{eqnarray*}
 {\mathbb{E}} \Bigl[\sup_{i}\bigl |A^N_{t^N_i}\bigr|^2
\Bigr] &\le& 2{\mathbb{E}} \Biggl[ \Biggl(\sum_{j=1}^{i_N}
\bigl|\tilde {Y}_{t^N_{j-1}}^N-Y^N_{t^N_{j-1}} \bigr|\Delta
\bigl\anglel W^N\bigr\angler_{t^N_{j}} \Biggr)^2
\Biggr]
\\[-2pt]
&\le& 2 \bigl\anglel W^N\bigr\angler_{T} {\mathbb{E}}
\Biggl[\sum_{j=1}^{i_N} \bigl|\tilde{Y}^N_{t^N_{j-1}}-Y^N_{t^N_{j-1}}\bigr|^2
\Delta\bigl\anglel W^N\bigr\angler_{t^N_{j}} \Biggr]
\\[-2pt]
&\le& 2 \bigl(\bigl\anglel W^N\bigr\angler_{T}
\bigr)^2 \sup_t {\mathbb{E}} \bigl[\bigl|\tilde
{Y}^N_{t}-Y^N_{t}\bigr|^2
\bigr] \rightarrow0\qquad \mbox{as } N \rightarrow\infty,
\end{eqnarray*}
where we used Jensen's inequality for the second inequality and
(\ref{l2}) for the convergence in the last line. This proves the theorem.
\end{pf}

If one has convergence of $(W^N,\xi^N)$ to $(W,\xi)$ in distribution
instead of $L^2$ together with
\[
\sup_N {\mathbb{E}} \Bigl[\sup_t \bigl|W^N_t\bigr|^{2 + \varepsilon}
\Bigr]<\infty\quad\mbox{and}\quad \sup_N \bigl\llVert \xi^N \bigr
\rrVert_{\infty} < \infty,\vadjust{\goodbreak}
\]
one can show as in Example~\ref{ExBernoulli} that there exists a
probability space
$(\tilde{\Omega}, \tilde\mathcal{F}, \tilde{\p})$ carrying
$(\tilde{W}^N,\tilde{\xi}^N)$
distributed as $(W^N,\xi^N)$ and $(\tilde{W},\tilde{\xi})$
distributed as $(W,\xi)$
such that
\[
{\mathbb{E}} \Bigl[\sup_t \bigl|\tilde{W}^N_t -
\tilde{W}_t\bigr|^2 \Bigr] \to0 \quad\mbox{and}\quad {\mathbb{E}} \bigl[\bigl|
\tilde{\xi}^N - \tilde{\xi}\bigr|^2 \bigr] \to0 \qquad\mbox{for } N
\to \infty.
\]
In the case where the drivers $f$ and $f^N$ are given as in
\eqref{ftildef} and \eqref{fNtildef}, the following holds.
%
%co6.3 #&#
\begin{corollary}
Assume the $W^N$ fulfill \emph{(C1)}, \emph{(C2)} and \emph{(W2)},
but instead of \emph{(C3)}, $(W^N, \xi^N)$ converges in distribution to
$(W,\xi)$
and one has $\sup_N {\mathbb{E}} [\sup_t |W^N_t|^{2+\varepsilon
} ] < \infty$
for some $\varepsilon> 0$
and $\sup_N \llVert \xi^N \rrVert_{\infty} < \infty$.
Furthermore, suppose $f$
and $f^N$ are of the form \eqref{ftildef} and
\eqref{fNtildef}, respectively. Then for $N$ large enough, the $N$th
BS$\Delta$E has a unique solution
$(Y^N,Z^N,M^N)$ and
\[
\sup_t \bigl|Y^N_t - Y_t\bigr| \to0\qquad
\mbox{in distribution for } N \to\infty,
\]
where $(Y,Z)$ is the unique solution of the BSDE \eqref{bsde} with
bounded $Y$.
\end{corollary}

\vspace*{-20pt}

%sA #&#
\begin{appendix}\label{app}
\setcounter{equation}{0}
\setcounter{theorem}{0}
\section*{Appendix}

%sA.1 #&#
\subsection{\texorpdfstring{Proofs of Section \protect\ref{secsol}}{Proofs of Section 3}}

\begin{pf*}{Proof of Lemma~\ref{lemmaYZM}}
If $(Y^N,Z^N, M^N)$ is a solution of
the $N$th BS$\Delta$E, then
%
%eA.1 #&#
\begin{equation}
\label{basis} Y^N_{t^N_i} - f^N
\bigl(t^N_{i+1},Y^N_{t^N_i},
Z^N_{t^N_{i+1}}\bigr)\Delta\bigl\anglel W^N\bigr
\angler_{t^N_{i+1}} + Z^N_{t^N_{i+1}} \Delta
W^N_{t^N_{i+1}} + \Delta M^N_{t^N_{i+1}} =
Y^N_{t^N_{i+1}} .
\end{equation}
Taking conditional expectations on
both sides with respect to $\F^N_{t^N_i}$ gives
\eqref{Yformula}. Multiplying both sides of \eqref{basis} with
$\Delta W^{N,k}_{t^N_{i+1}}$ and taking conditional
expectations with respect to $\F^N_{t^N_i}$ yields
\eqref{Zformula}. Finally, \eqref{Mformula} is a consequence of
\eqref{Yformula} and \eqref{basis}.
\end{pf*}

\begin{pf*}{Proof of Proposition~\ref{propex}}
We prove the proposition by backwards induction. Set $Y^N_T =
\xi^N$, which by assumption (C4) is bounded. Now assume that
there exist $i$ and $(Y^N_t, Z^N_t, M^N_t)$ solving the
BS$\Delta$E \eqref{bdisc} for $t \in[t^N_{i+1},T]$ such that
$(Y^N_t)$ and $(Z^N_t)$ are bounded. By Lemma~\ref{lemmaYZM},
$Z^{N,k}_{t^N_{i+1}}$ must be of the form
\[
Z^{N,k}_{t^N_{i+1}} = \frac{{\mathbb{E}} [Y^N_{t^N_{i+1}} \Delta
W^{N,k}_{t^N_{i+1}}|\F^N_{t^N_i} ]}{\Delta\anglel W^{N}\angler_{t^N_{i+1}}}.
\]
Since by induction hypothesis, $Y^N_{t^N_{i+1}}$ is bounded,
$Z^{N,k}_{t^N_{i+1}}$ is well-defined and bounded.
Next, we try to find $Y^N_{t^N_i} \in L^{\infty}(\mathcal
{F}^N_{t^N_i})$ such that
%
%eA.2 #&#
\begin{equation}
\label{Ysolvi} Y^N_{t^N_i} - f^N
\bigl(t^N_{i+1},Y^N_{t^N_i},Z^N_{t^N_{i+1}}
\bigr) \Delta\bigl\anglel W^N\bigr\angler_{t^N_{i+1}}={\mathbb{E}}
\bigl[Y^N_{t^N_{i+1}}|\F^N_{t^N_i} \bigr].
\end{equation}
To do that, we introduce the mapping
$A(\omega,y) := y - f(t^N_{i+1},y,Z^N_{t^N_{i+1}}) \Delta\anglel W^N\angler_{t^N_{i+1}}$. It is
$\mathcal{F}^N_{t^N_i} $-measurable in $\omega$ and continuous in
$y$. Moreover, it satisfies
%
%eA.3 #&#
\begin{equation}
\label{Ay} y - \kappa\bigl(1+|y|+ g\bigl(Z^N_{t^N_{i+1}}\bigr)
\bigr) \le A(\omega,y) \le y + \kappa\bigl(1+|y|+ g\bigl(Z^N_{t^N_{i+1}}
\bigr)\bigr)
\end{equation}
for $\kappa= K \max_i \Delta\anglel W^N\angler_{t^N_i} < 1$. So it
follows from Lemma~\ref{lemmaA} below that there
exists an $\F^N_{t^N_i} \otimes\mathcal{B}(\mathbb{R})$-measurable
function $B \dvtx \Omega\times\mathbb{R}
\rightarrow\mathbb{R}$ such that $A(\omega,B(\omega,y)) = y$
for all $(\omega,y) \in\Omega\times\mathbb{R}$. Thus,
\[
Y^N_{t^N_i} = B \bigl(\omega,{\mathbb{E}}
\bigl[Y^N_{t^N_{i+1}}|\F^N_{t^N_i} \bigr]
\bigr) \in L^0\bigl(\mathcal{F}^N_{t^N_i}\bigr)
\]
solves \eqref{Ysolvi}, and since $Y^N_{t^N_{i+1}}$ and $Z^N_{t^N_{i+1}}$
are bounded, it follows from the estimate
\eqref{Ay} that the same is true for $Y^N_{t^N_i}$. Finally, $M^N_0
= 0$ and
\begin{eqnarray*}
\Delta M^N_{t^N_{i+1}} &=& \Delta Y^N_{t^N_{i+1}}
+ f^N\bigl(t^N_{i+1},Y^N_{t^N_i},Z^N_{t^N_{i+1}}
\bigr) \Delta\bigl\anglel W^N\bigr\angler_{t^N_{i+1}} -
Z^N_{t^N_{i+1}} \Delta W^N_{t^N_{i+1}}
\\
&=& Y^N_{t^N_{i+1}} - {\mathbb{E}} \bigl[Y^N_{t^N_{i+1}}|
\F^N_{t^N_i} \bigr] - Z^N_{t^N_{i+1}} \Delta
W^N_{t^N_{i+1}}
\end{eqnarray*}
defines a square-integrable martingale $M^N$ orthogonal to $W^N$ which
is bounded if
$W^N$ is so. This completes the proof.
\end{pf*}

%leA.1 #&#
\begin{lemma} \label{lemmaA}
Let $\mathcal{G}$ be a sub-$\sigma$-algebra of $\mathcal{F}$ and $A
\dvtx  \Omega\times\mathbb{R} \to\mathbb{R}$ a
function that is $\mathcal{G}$-measurable in $\omega\in\Omega$ and
continuous in
$y \in\mathbb{R}$. Assume that for every $\omega\in\Omega$, the
set $ \{y \in\mathbb{R} \dvt  A(\omega,y) \in C \}$ is
nonempty and bounded for each nonempty bounded subset
$C$ of $\mathbb{R}$. Then there exists a $\mathcal{G} \otimes
\mathcal{B}(\mathbb{R})$-measurable function $B \dvtx  \Omega\times
\mathbb{R} \to\mathbb{R}$ such that $A(\omega,B(\omega,x)) =
x$ for all $x \in\mathbb{R}$.
\end{lemma}

\begin{pf}
For all $k,l \in\mathbb{N}$,
\[
b_{kl}(\omega) = \inf \bigl\{y \in\mathbb{R} \dvt  A(\omega,y)
\in\bigl((k-1)2^{-l}, k 2^{-l}\bigr] \bigr\}
\]
is a $\mathcal{G}$-measurable mapping from $\Omega$ to
$\mathbb{R}$ and
\[
B_l(\omega,x) = \sum_{k \in\mathbb{Z}}
b_{kl}(\omega) 1_{ \{(k-1) 2^{-l} < x \le k 2^{-l} \}}
\]
a $\mathcal{G} \otimes\mathcal{B}(\mathbb{R})$-measurable map from
$\Omega\times\mathbb{R}$ to $\mathbb{R}$ such that
\[
B_l(\omega,x) \to B(\omega,x)\qquad \mbox{as } l \to\infty
\]
for a $\mathcal{G} \otimes\mathcal{B}(\mathbb{R})$-measurable
function $B : \Omega\times\mathbb{R} \to\mathbb{R}$. Since
$y \mapsto A(\omega,y)$ is continuous for all $\omega\in
\Omega$, one obtains
\[
A\bigl(\omega, B(\omega,x)\bigr) = A\Bigl(\omega, \lim_{l\to\infty}
B_l(\omega,x)\Bigr) = \lim_{l\to\infty} A\bigl(\omega,
B_l(\omega,x)\bigr) = x
\]
for all $x \in\mathbb{R}$.
\end{pf}

%sA.2 #&#
\subsection{\texorpdfstring{Proofs of Lemmas \protect\ref{lemmacomp1} and \protect\ref{lemmacomp2}}
{Proofs of Lemmas 4.3 and 4.4}}

To prove Lemmas~\ref{lemmacomp1} and~\ref{lemmacomp2}, we need the following
lemma.
%
%leA.2 #&#
\begin{lemma} \label{lemmadetex}
Assume the $N$th driver and terminal condition are of the special form
\[
f^N(t,y,z) = K\bigl(1 + \llvert y\rrvert + g(z)\bigr)\quad \mbox{and}\quad
\xi^N = C
\]
for constants $C,K \in\mathbb{R}_+$ and a measurable function
$g \dvtx  \mathbb{R}^d \to\mathbb{R}$ with $g(0) = 0$. Then for all
$N \in\mathbb{N}$ such that $\max_i \Delta\anglel W^N\angler_{t^N_i} < 1/K$,
the $N$th BS$\Delta$E has a unique solution $(Y^N,Z^N,M^N)$ given by
%
%eA.4 #&#
\begin{equation}
\label{detsol} Y^N_T = C ,\qquad Y^N_{t^N_i}
= \frac
{Y^N_{t^N_{i+1}} +
K \Delta\anglel W^N\angler_{t^N_{i+1}}}{1 - K
\Delta\anglel W^N\angler_{t^N_{i+1}}} ,\qquad Z^N_{t^N_i} = 0 ,\qquad M^N_{t^N_i}
= 0.
\end{equation}
In particular, $Y^N$ is deterministic and for $N \to\infty$,
converges uniformly to
the function
\[
(C+1) \exp\bigl(K(T-t)\bigr) - 1.
\]
\end{lemma}

\begin{pf}
Since the terminal condition and the increments $\Delta\anglel W^N\angler_{t^N_i}$ are deterministic,
$Z^N$ and $M^N$ are both zero and $Y^N$ solves
%
%eA.5 #&#
\begin{equation}
\label{deteq} Y^N_{t^N_i} = Y^N_{t^N_{i+1}} +
K\bigl(1+ \bigl|Y^N_{t^N_i}\bigr|\bigr) \Delta \bigl\anglel
W^N\bigr\angler_{t^N_{i+1}} ,\qquad Y^N_T = C.
\end{equation}
This shows \eqref{detsol}. Moreover, since \eqref{deteq} are
deterministic difference equations with Lipschitz coefficients,
one obtains from Theorem~\ref{thmBriand} that their solutions converge
uniformly to the solution of the ordinary
differential equation
\[
y'(t) = - K\bigl(1+ \bigl|y(t)\bigr|\bigr) ,\qquad y(T) = C,
\]
given by
%
%eA.6 #&#
\begin{eqnarray}
y(t) = (C+1) \exp\bigl(K(T-t)\bigr) -1.\nonumber\\[-17pt]
\eqntext{\qed}\end{eqnarray}
\noqed\end{pf}

\begin{pf*}{Proof of Lemma~\ref{lemmacomp1}}
Since $\max_i \Delta
\anglel W^N\angler_{t^N_i} < 1/K$ for $N$ large enough, it follows from
Lemma~\ref{lemmadetex} that there exists an $N_1 \ge1$ such
that for all $N \ge N_1$, the BS$\Delta$E with driver
$\hat{f}^N(t,y,z) = K(1 + \llvert y\rrvert  + \llvert z\rrvert^q)$ and terminal
condition $\hat{\xi}^N = C$ has a deterministic solution
$\hat{Y}^N$ that is bounded by $(C+1) \exp(K(T-t))$. Set
$D = 2(C+1) \exp(KT)$. Since $q < 2$, one obtains from condition
(C5) that there exists $N_0 \ge N_1$ such that
%
%eA.7 #&#
\begin{equation}
\label{choose1} \sup_{N \ge N_0} \max_i K \bigl(1+
d^{q/2} D^q \bigl[\Delta\bigl\anglel W^N\bigr
\angler_{t^N_i} \bigr]^{- q/2} \bigr) \Delta \bigl\anglel
W^N\bigr\angler_{t^N_i} < 1
\end{equation}
and
%
%eA.8 #&#
\begin{equation}
\label{choose2} \sup_{N \ge
N_0} \max_{i,k} d q K \bigl(1 +
d^{q/4} D^{q/2} \bigl[\Delta\bigl\anglel W^N\bigr
\angler_{t^N_i} \bigr]^{- q/4} \bigr) \bigl\llVert \Delta
W^{N,k}_{t^N_i} \bigr\rrVert_{\infty} \le1.
\end{equation}
Fix $N \ge N_0$ and let $f^N_1 \ge f^N_2$ be drivers
and $\xi^N_1 \ge\xi^N_2$ terminal conditions satisfying assumptions (i)--(iv)
of Lemma~\ref{lemmacomp1}. By Proposition~\ref{propex}, both
BS$\Delta$Es have
a solution $(Y^N_m, Z^N_m, M^N_m)$, $m=1,2$, and
\eqref{statlemma} clearly holds at the final time $T$. We now
go backwards in time and assume \eqref{statlemma} is true on
$[t_{i+1}, T]$. Then
%
%eA.9 #&#
\begin{equation}
\label{boundYi1} (C+1) \exp\bigl(K\bigl(T-{t^N_{i+1}}\bigr)
\bigr) \ge Y^N_{1,t^N_{i+1}} \ge Y^N_{2,t^N_{i+1}}
\ge- (C+1) \exp\bigl(K\bigl(T-{t^N_{i+1}}\bigr)\bigr).
\end{equation}
By
Lemma~\ref{lemmaYZM}, one has
%
%eA.10 #&#
\begin{equation}
\label{Zform} Z^{N,k}_{m,t^N_{i+1}} = \frac{{\mathbb{E}} [Y^N_{m,t^N_{i+1}}
\Delta W^{N,k}_{t^N_{i+1}} | \mathcal{F}^N_{t^N_i} ]}{\Delta
\anglel W^{N}\angler_{t^N_{i+1}}},
\end{equation}
and
\[
Y^N_{m,t^N_i} = {\mathbb{E}} \bigl[Y^N_{m,t^N_{i+1}}|
\mathcal {F}^N_{t^N_i} \bigr] +f^N_m
\bigl(t^N_{i+1},Y^N_{m,t^N_i},Z^N_{m,t^N_{i+1}}
\bigr) \Delta\bigl\anglel W^N\bigr\angler_{t^N_{i+1}}.
\]
Set
\[
Y^N_t := Y^N_{1,t} -
Y^N_{2,t} ,\qquad Z^N_t :=
Z^N_{1,t} - Z^N_{2,t}.
\]
By \eqref{boundYi1}, $Y^N_{1,t^N_{i+1}}$, $Y^N_{2,t^N_{i+1}}$
and $Y^N_{t^N_{i+1}}$ are bounded by $D$ and
\[
Y^N_{t^N_i} = {\mathbb{E}} \bigl[Y^N_{t^N_{i+1}}
\mid\mathcal {F}^N_{t^N_{i+1}} \bigr] + \bigl(\alpha+
Y^N_{t^N_i} \beta+ Z^N_{t^N_{i+1}} \gamma
\bigr) \Delta\bigl\anglel W^N\bigr\angler_{t^N_{i+1}}
\]
for
\begin{eqnarray*}
\alpha&=& f^N_1\bigl(t^N_{i+1},Y^N_{2,t^N_i},Z^N_{2,t^N_{i+1}}
\bigr) - f^N_2\bigl(t^N_{i+1},Y^N_{2,t^N_i},Z^N_{2,t^N_{i+1}}
\bigr),
\\
\beta&=& \frac{1}{Y^N_{t^N_i}} \bigl(f^N_1
\bigl(t^N_{i+1},Y^N_{1,t^N_i},Z^N_{2,t^N_{i+1}}
\bigr) - f^N_1\bigl(t^N_{i+1},Y^N_{2,t^N_i},Z^N_{2,t^N_{i+1}}
\bigr) \bigr),
\\
\gamma^k &=& \frac{1}{Z^{N,k}_{t_{i+1}}} \bigl( f^N_1
\bigl(t^N_{i+1},Y^N_{1,t^N_i},Z^{N,1}_{1,t^N_{i+1}},
\ldots, Z^{N,k}_{1,t^N_{j+1}},Z^{N,k+1}_{2,t^N_{j+1}},\ldots,
Z^{N,d}_{2,t^N_{j+1}}\bigr)
\\
&&\hspace*{26.5pt} {}- f^N_1 \bigl(t^N_{i+1},Y^N_{1,t^N_i},Z^{N,1}_{1,t^N_{i+1}},
\ldots, Z^{N,k-1}_{1,t^N_{i+1}},Z^{N,k}_{2,t^N_{i+1}},\ldots,
Z^{N,d}_{2,t^N_{i+1}}\bigr) \bigr).
\end{eqnarray*}
It can be seen from
\eqref{Zform} that for $m=1,2$,
%
%eA.11 #&#
\begin{equation}
\label{boundZm} \hspace*{-10pt}\bigl\llvert Z^N_{m,t^N_{i+1}}\bigr
\rrvert^2 = \sum_{k=1}^d
\bigl(Z^{N,k}_{m,t^N_{i+1}} \bigr)^2 \le\sum
_{k=1}^d \frac{{\mathbb{E}} [(Y^N_{m,t^N_{i+1}})^2 \mid\mathcal
{F}^N_{t^N_i} ]
{\mathbb{E}} [(\Delta W^{N,k}_{t^N_{i+1}})^2 ]}{(\Delta
\anglel W^N\angler_{t^N_{t^N_{i+1}}})^2} \le\frac{d D^2}{\Delta
\anglel W^N\angler_{t^N_{i+1}}}.
\end{equation}
So by assumption (iii) and
\eqref{choose1},
\begin{eqnarray*}
 \bigl\llvert \beta\Delta\bigl\anglel W^N\bigr\angler_{t^N_{i+1}}
\bigr\rrvert &\le& K\bigl(1 + \bigl|Z^N_{2,t^N_{i+1}}\bigr|^q\bigr)
\Delta\bigl\anglel W^N\bigr\angler_{t^N_{i+1}}
\\
&\le& K \bigl(1+ d^{q/2} D^q \bigl[\Delta\bigl\anglel
W^N\bigr\angler_{t^N_{i+1}}\bigr]^{- q/2} \bigr) \Delta
\bigl\anglel W^N\bigr\angler_{t^N_{i+1}} < 1.
\end{eqnarray*}
Hence,
%
%eA.12 #&#
\begin{equation}
\label{repY} Y^N_{t^N_i} = \frac{ {\mathbb{E}}
[Y^N_{t^N_{i+1}}| \mathcal{F}^N_{t^N_i} ] + (\alpha+
Z^N_{t^N_{i+1}} \gamma) \Delta
\anglel W^N\angler_{t^N_{i+1}}}{1 - \beta\Delta
\anglel W^{N}\angler_{t^N_{i+1}}}.
\end{equation}
From assumption (iv) and
\eqref{boundZm} one obtains
\begin{eqnarray*}
|\gamma| \le d^{1/2} q K \bigl(1 + \bigl(\bigl\llvert
Z^N_{1,t^N_{i+1}}\bigr\rrvert \vee\bigl\llvert Z^N_{2,t^N_{i+1}}
\bigr\rrvert \bigr)^{q/2} \bigr) \le d^{1/2} q K \bigl(1 +
d^{q/4} D^{q/2} \bigl(\Delta\bigl\anglel W^N\bigr
\angler_{t^N_{i+1}}\bigr)^{-
q/4} \bigr)
\end{eqnarray*}
and from \eqref{Zform},
\[
\bigl|Z^N_{t^N_{i+1}}\bigr| \le d^{1/2} \max_k
\frac{\llVert \Delta W^{N,k}_{t^N_{i+1}} \rrVert_{\infty}}{\Delta
\anglel W^N\angler_{t^N_{i+1}}} {\mathbb{E}} \bigl[\bigl|Y^N_{t^N_{i+1}}\bigr| \mid
\mathcal{F}^N_{t^N_i} \bigr].
\]
By \eqref{choose2}, this yields
\begin{eqnarray*}
\bigl|Z^N_{t^N_{i+1}} \gamma\bigr| \Delta\bigl\anglel W^{N}
\bigr\angler_{t^N_{i+1}}& \le&\bigl|Z^N_{t^N_{i+1}}\bigr | |\gamma| \Delta
\bigl\anglel W^{N}\bigr\angler_{t^N_{i+1}}
\\
&\le& d q K \bigl(1 + d^{q/4} D^{q/2} \bigl(\Delta\bigl\anglel
W^N\bigr\angler_{t^N_{i+1}}\bigr)^{- q/4} \bigr)
\max_k \bigl\llVert \Delta W^{N,k}_{t^N_{i+1}} \bigr
\rrVert_{\infty} {\mathbb{E}} \bigl[\bigl|Y^N_{t^N_{i+1}}\bigr| \mid
\mathcal{F}^N_{t^N_i} \bigr]
\\
&\le &{\mathbb{E}} \bigl[\bigl|Y^N_{t^N_{i+1}}\bigr| \mid\mathcal
{F}^N_{t^N_i} \bigr].
\end{eqnarray*}
Since $Y^N_{t^N_{i+1}} \ge0$ and $\alpha\ge0$, it follows
from \eqref{repY} that $Y^N_{1,t^N_i} - Y^N_{2,t^N_i} =
Y^N_{t^N_i} \ge0$. Now observe that $\hat{f}^N$ satisfies
assumptions (ii)--(iv). So the same argument applied to the
equations corresponding to $(\hat{f}^N,C)$ and $(f^N_1,\xi^N)$ gives
\[
(C+1) \exp\bigl(K\bigl(T-{t^N_i}\bigr)\bigr) \ge
\hat{Y}^N_{t^N_i} \ge Y^N_{1,t^N_i}.
\]
Analogously, one deduces
\[
Y^N_{2,t^N_i} \ge(C+1) \exp\bigl(K\bigl(T-{t^N_i}
\bigr)\bigr),
\]
and the induction step is complete.
\end{pf*}

\begin{pf*}{Proof of Lemma~\ref{lemmacomp2}}
For $N$ large enough, one has
%
%eA.13 #&#
\begin{equation}
\label{deltaK} \max_i \Delta\bigl\anglel W^N\bigr
\angler_{t^N_i} < 1/K.
\end{equation}
So it follows from Lemma~\ref{lemmadetex} that there exists
$N_1 \in\mathbb{N}$ such that for all $N \ge N_1$, the BS$\Delta$E with
driver $\hat{f}^N(t,y,z) = K(1 + \llvert y\rrvert  + \llvert z\rrvert^q)$ and
terminal condition $\hat{\xi}^N = C$ has a deterministic
solution $\hat{Y}^N$ dominated by $(C+1) \exp(K(T-t))$.
Choose $N_0 \geq N_1$ such that for all $N \ge N_0$, the
statement of Lemma~\ref{lemmacomp1} holds for all terminal
conditions bounded by $(C+1) \exp(KT)$ and drivers satisfying
conditions (ii)--(iv) of Lemma~\ref{lemmacomp1}. Now fix $N \ge
N_0$ and assume $(Y^N, Z^N, M^N)$ is a solution
corresponding to $\xi^N$ and $f^N$ satisfying conditions (i)
and (ii) of Lemma~\ref{lemmacomp2}. Since $\hat{Y}^N_t \le(C+1) \exp
(K(T-t))$, it is
enough to show that
%
%eA.14 #&#
\begin{equation}
\label{statlemma3} \hat{Y}^N_{t^N_i} \ge Y^N_{t^N_i}
\ge- \hat{Y}^N_{t^N_i}\qquad \mbox{for all } i.
\end{equation}
By condition (i), \eqref{statlemma3} holds for $t=T$.
For $t < T$ we argue by backwards induction. So let us assume
that \eqref{statlemma3} holds for $t = t^N_{i+1}$. We will only
show $\hat{Y}^N_{t^N_i} \ge Y^N_{t^N_i}$. The second inequality
in (\ref{statlemma3}) follows analogously. From Lemma
\ref{lemmaYZM}, we know that
\[
Z^{N,k}_{t^N_{i+1}} = \frac{{\mathbb{E}} [Y^N_{t^N_{i+1}} \Delta
W^{N,k}_{t^N_{i+1}} |\mathcal{F}^N_{t^N_i} ]}{\Delta\anglel W^N\angler_{t^N_{i+1}}}
\]
and
\[
A\bigl(\omega,Y^N_{t^N_i}\bigr) = {\mathbb{E}}
\bigl[Y^N_{t^N_{i+1}} | \mathcal {F}^N_{t^N_i}
\bigr],
\]
where $A(\omega,y) = y-f(t^N_{i+1},y,Z^N_{t^N_{i+1}}) \Delta
\anglel W^{N}\angler_{t^N_{i+1}}$. Consider the BS$\Delta$E with driver
\[
\tilde{f}^N\bigl(t^N_j,y,z\bigr) = \cases{
K\bigl(1 + \llvert y\rrvert + \llvert z\rrvert^q\bigr), &\quad for $j =
i+1$,
\cr
0, &\quad for $j \neq i+1$ }
\]
and terminal condition $Y^N_{t^N_{i+1}}$. By Lemma
\ref{lemmacomp1}, it has a unique solution $(\tilde{Y}^N,\tilde
{Z}^N, \tilde{M}^N)$, and it is easy to see that
$\tilde{Y}^N_{t^N_{i+1}} = Y^N_{t^N_{i+1}}$. Due to
\eqref{deltaK}, the mapping $\tilde{A}(\omega,y) = y-
\tilde{f}(t^N_{i+1},y,Z^N_{t^N_{i+1}}) \*\Delta
\anglel W^{N}\angler_{t^N_{i+1}}$ is strictly increasing in $y$ and
since $\tilde{f}^N(t^N_{i+1},\cdot\,,\cdot) \ge f^N(t^N_{i+1},\cdot\,,\cdot)$, one
has
\[
\tilde{A} \bigl(\omega, \tilde{Y}^N_{t^N_i}\bigr) = {
\mathbb{E}} \bigl[Y^N_{t^N_{i+1}} | \mathcal{F}^N_{t^N_i}
\bigr] = A \bigl(\omega, Y^N_{t^N_i}\bigr) \ge\tilde{A}\bigl(
\omega, Y^N_{t^N_i}\bigr).
\]
This shows $\tilde{Y}^N_{t^N_i} \ge Y^N_{t^N_i}$. To conclude
the proof, consider the solution $\bar{Y}^N$ of the
BS$\Delta$E with driver $\tilde{f}^N$ and terminal condition
$\hat{Y}^N_{t^N_{i+1}}$. Then $\bar{Y}^N_{t^N_i} =
\hat{Y}^N_{t^N_i}$ and Lemma~\ref{lemmacomp1} yields
$\bar{Y}^N_{t^N_i} \ge\tilde{Y}^N_{t^N_i}$. Consequently,
\[
\hat{Y}^N_{t^N_i} = \bar{Y}^N_{t^N_i} \ge
\tilde{Y}^N_{t^N_i} \ge Y^N_{t^N_i},
\]
which completes the induction step.
\end{pf*}

%sA.3 #&#
\subsection{\texorpdfstring{Remaining proofs of Section \protect\ref{secconvergence}}
{Remaining proofs of Section 5}}

\begin{pf*}{Proof of Lemma~\ref{infterm2}}
Set $\tilde{C} = 3C$ and $\tilde{K}= 2K(2C+K+1) (\exp(KT)+1)(T+1)$.
Choose $b \in
\mathbb{R}_+$ such that condition (f3) holds for $a =
(\tilde{C}+1) \exp(\tilde{K}T)$. It follows from \eqref{WNt}
that $\prod_{i=1}^{i_N} (1-K\Delta\anglel W^N\angler_{t^N_i}) \to
\exp
(-KT)$ for $N \to\infty$. So there exists $N_0 \in\mathbb{N}$
such that for all $N \ge N_0$,
\[
\prod_{i=1}^{i_N} \bigl(1- K \Delta\bigl
\anglel W^N\bigr\angler_{t^N_i}\bigr)^{-1} \le
\exp(KT) + 1,\qquad \bigl\anglel W^N\bigr\angler_T\le T+1
\]
and the statement of Theorem~\ref{thmcomp} holds for
$\tilde{C}$ instead of $C$, $\tilde{K}$ instead of $K$ and $L =
K \vee b$. Set $D = (\exp(KT)+ 1)(T+1)$ and fix $N \ge N_0$ as
well as terminal conditions $\xi^N_1$, $\xi^N_2$ bounded by $C$
and drivers $f^N_1$, $f^N_2$ satisfying (f1)--(f3) such
that $\Vert f^N_1-f^N_2\Vert_\infty\le K$. Then the parameter pairs
$(f^N_m, \xi^N_m)$, $m=1,2$, and $(\tilde{f}^N,\tilde{\xi}^N)$,
where $\tilde{f}^N = f^N_2 + \Vert f^N_1-f^N_2\Vert_\infty$ and
$\tilde{\xi}^N = \xi^{N}_2+\Vert\xi^N_1-\xi^N_2\Vert_\infty$, satisfy
the conditions of Theorem~\ref{thmcomp} for $\tilde{C}$ instead
of $C$, $\tilde{K}$ instead of $K$ and $L = K \vee b$.
Therefore, the corresponding BS$\Delta$Es have unique
solutions, which, since $\tilde{f}^N \ge f^N_1$ and
$\tilde{\xi}^N \ge\xi^N_1$, satisfy $\tilde{Y}^N_t \ge
Y^N_{1,t}$ for all $t$. Note that the solution of the
deterministic BS$\Delta$E
%
%eA.15 #&#
\begin{eqnarray}
\label{BSDEhat} %
\hat{Y}^N_{t^N_i} &=&
\hat{Y}^N_{t^N_{i+1}} +\bigl(\bigl\Vert f^N_1-f^N_2\bigr\Vert_\infty+K
\hat{Y}^N_{t^N_i}\bigr)\Delta\bigl\anglel W^N\bigr
\angler_{t^N_{i+1}},\nonumber
\\[-9pt]\\[-9pt]
\hat{Y}^N_T &=& \bigl\Vert\xi^N_1-
\xi^N_2\bigr\Vert_\infty,
\nonumber
\end{eqnarray}
is given by
\begin{eqnarray*}
\hat{Y}^N_{t^N_i}= \frac{\Vert\xi^N_1-\xi^N_2\Vert_\infty}{\prod_{j=i+1}^{i_N}
(1-K\Delta\anglel W^N\angler_{t^N_j})} +\bigl\Vert f^N_1-f^N_2\bigr\Vert_\infty
\sum_{j=i+1}^{i_N} \frac{ \Delta
\anglel W^N\angler_{t^N_j}}{\prod_{l=i+1}^j
(1-K\Delta\anglel W^N\angler_{t^N_l})}.
\end{eqnarray*}
In particular,
$\hat{Y}^N_t$ is positive and decreasing in $t$, and it
satisfies
\[
\hat{Y}^N_{t^N_i} \le \frac{\Vert\xi^N_1-\xi^N_2\Vert_\infty+
\Vert f^N_1-f^N_2\Vert_\infty\sum_{j=i+1}^{i_N} \Delta
\anglel W^N\angler_{t^N_j}}{\prod_{j=i+1}^{i_N} (1-K\Delta\anglel W^N\angler_{t^N_j})}.
\]
Hence, by the choice of the constant $D$, one obtains the
estimate
%
%eA.16 #&#
\begin{equation}
\label{estD} \sup_t \hat{Y}^N_t=
\hat{Y}^N_0 \le D\bigl(\bigl\Vert\xi^N_1-
\xi^N_2\bigr\Vert_\infty+\bigl\Vert f^N_1-f^N_2\bigr\Vert_\infty
\bigr).
\end{equation}
In
particular, since $\Vert\xi^N_1-\xi^N_2\Vert_\infty\le2C$ and
$\Vert f^N_1-f^N_2\Vert_\infty\le K$, it follows from \eqref{estD}
that
%
%eA.17 #&#
\begin{equation}
\label{hatYbound} \sup_t \hat{Y}^N_t \le
(2C+K) \bigl(\exp(KT) + 1\bigr) (T+1).
\end{equation}
Next, notice that the process
\[
\bar{Y}^N_t :=Y^{N}_{2,t} +
\hat{Y}^N_t
\]
satisfies
\begin{eqnarray*}
\bar{Y}^N_{t^N_i} &=& \bar{Y}^N_{t^N_{i+1}} +
\bigl\{f^N_2\bigl(t^N_{i+1},W^N,
Y^{N}_{2,t^N_i},Z^{N}_{2,t^N_{i+1}}\bigr) +
\bigl\Vert f^N_1-f^N_2\bigr\Vert_\infty+K
\hat{Y}^N_{t^N_i} \bigr\} \Delta\bigl\anglel W^N
\bigr\angler_{t^N_{i+1}}
\\[-2pt]
&& {}-Z^{N}_{2,t^N_{i+1}}\Delta W^N_{t^N_{i+1}}-
\Delta M^{N}_{2,t^N_{i+1}},
\\[-2pt]
\bar{Y}^N_T&=&\xi^{N}_2 + \bigl\Vert
\xi^N_1-\xi^N_2\bigr\Vert_\infty,
\end{eqnarray*}
and
since $f^N_2$ is $K$-Lipschitz in $y$, one has
\[
f^N_2\bigl(t^N_{i+1},W^N,
\bar{Y}^{N}_{t^N_i},Z^{N}_{2,t^N_{i+1}}\bigr) \le
f^N_2\bigl(t^N_{i+1},W^N,Y^{N}_{2,t^N_i},
Z^N_{2,t^N_{i+1}}\bigr)+K\hat{Y}^N_{t^N_i}.
\]
Hence,
\[
\alpha_{t^N_i} = f^N_2\bigl(t^N_{i+1},
W^N,Y^N_{2,t^N_i}, Z^N_{2,t^N_{i+1}}
\bigr) - f^N_2\bigl(t^N_{i+1},W^N,
\bar{Y}^N_{t^N_i},Z^{N}_{2,t^N_{i+1}}\bigr) + K
\hat{Y}^N_{t^N_i} \ge0
\]
and $\bar{Y}^N$ satisfies the BS$\Delta$E
%
%eA.18 #&#
\begin{eqnarray}
\label{BSDEYbar} %
\bar{Y}^{N}_{t^N_i} &=&
\bar{Y}^N_{t^N_{i+1}} + \bigl\{f^N_2
\bigl(t^N_{i+1}, W^N, \bar{Y}^N_{t^N_i},Z^N_{2,t^N_{i+1}}
\bigr) + \bigl\Vert f^N_1-f^N_2\bigr\Vert_\infty+
\alpha_{t^N_i} \bigr\} \Delta\bigl\anglel W^N\bigr
\angler_{t^N_{i+1}}
\nonumber
\\
&&{} - Z^N_{2,t^N_{i+1}} \Delta W^N_{t^N_{i+1}} -
\Delta M^N_{2,t^N_{i+1}},
\\
\bar{Y}^N_T &=& \xi^N_2 + \bigl\Vert
\xi^N_1-\xi^N_2\bigr\Vert_\infty.
\nonumber
\end{eqnarray}
Since $f^N_2$ is $K$-Lipschitz in $y$, one obtains from the
estimate \eqref{hatYbound} that
\begin{eqnarray*}
\llVert \alpha_{t^N_i} \rrVert_\infty &\le& \bigl\Vert f^N_2\bigl(t^N_{i+1},
W^N,Y^N_{2,t^N_i}, Z^N_{2,t^N_{i+1}}
\bigr) - f^N_2\bigl(t^N_{i+1},
W^N,\bar{Y}^N_{t^N_i},Z^N_{2,t^N_{i+1}}
\bigr) \bigr\Vert_\infty + K \bigl\llVert \hat{Y}^N_{t^N_i}
\bigr\rrVert_\infty
\\
&\le& 2K \bigl\llVert \hat{Y}^N_{t^N_i} \bigr
\rrVert_\infty\le2K (2C + K) \bigl(\exp(KT)+1\bigr) (T+1),
\end{eqnarray*}
which shows that the BS$\Delta$E
\eqref{BSDEYbar} satisfies the assumptions of Theorem
\ref{thmcomp} for $\tilde{C}$, $\tilde{K}$ and $L = K \vee b$.
Hence, a comparison of $\tilde{Y}^N$ to $\bar{Y}^N$ yields
\[
Y^N_{1,t} \le\tilde{Y}^N_t \le
\bar{Y}^N_t = Y^N_{2,t}+
\hat{Y}^N_t \le Y^N_{2,t}+D \bigl(\bigl\Vert
\xi^N_1-\xi^N_2\bigr\Vert_\infty+
\bigl\Vert f^N_1-f^N_2\bigr\Vert_\infty
\bigr)
\]
for all $t$. By symmetry, one also has
\[
Y^{N}_{2,t} \le Y^{N}_{1,t} + D
\bigl(\bigl\Vert f^N_1-f^N_2\bigr\Vert_\infty+
\bigl\Vert\xi^N_1-\xi^N_2\bigr\Vert_\infty
\bigr)
\]
for all $t$, and the proof is complete.
\end{pf*}

\begin{pf*}{Proof of Lemma~\ref{l3b}}
Let $C \in\mathbb{R}_+$ such
that $\varphi$ is bounded by $C$ and $|\varphi(w_1) -
\varphi(w_2)| \le C \sup_{1 \le i \le n} |w_1(s_i) - w_2(s_i)|$
for all $w_1,w_2 \in\mathbb{R}^{d \times n}$. Choose $N_0 \in
\mathbb{N}$ and $D \in\mathbb{R}_+$ such that for all $N \ge
N_0$, $\sup_i |\Delta W^N_{t^N_i}| \le1$ and the statement of
Lemma~\ref{infterm2} holds. From Lemma~\ref{lemmaYZM}, we know that
\[
Z^{N,k}_{t^N_i}=\frac{
{\mathbb{E}} [Y^N_{t^N_i}\Delta W^{N,k}_{t^N_i}|\F^N_{t^N_{i-1}} ]}{\Delta\anglel W^N\angler_{t^N_i}},
\]
and since $Y^N_{t^N_i}$ is $\F^N_{t^N_i}$-measurable, it can be
written as
\[
Y^N_{t^N_i}= y^N_i
\bigl(W^N_{t^N_{1}}, \ldots, W^N_{t^N_i}\bigr)
\]
for a Borel measurable function $y^N_i \dvtx  \mathbb{R}^{d \times
i} \to\mathbb{R}$. We want to show that $y^N_i$ can be chosen
uniformly Lipschitz-continuous in the last argument. To do that,
let us condition on $W^N_{t_j} = w(t^N_j)$, $j = 1, \ldots, i-1$
and $W^N_{t^N_i} = x$. Denote $\tilde{W}^N_t =
W^N_t-W^N_{t^N_i}$, $t\in[t^N_i,T]$, and define $r = \max\{m \dvt
s_m \le t^N_i\}$. Then for $t^N_j \ge t^N_i$, the conditioned
BS$\Delta$E with solution $(Y^{N,x},Z^{N,x},M^{N,x})$ can
be written as
%
%eA.19 #&#
\begin{eqnarray}
\label{condBSDE} %
Y^{N,x}_{t^N_j} &=&
Y^{N,x}_{t^N_{j+1}}+ f^N\bigl(t^N_{j+1},w
\bigl(t^N_1\bigr), \ldots,w\bigl(t^N_{i-1}
\bigr), x + \tilde{W}^N, Y^{N,x}_{t^N_j},Z^{N,x}_{t^N_{j+1}}
\bigr)\Delta\bigl\anglel \tilde{W}^N\bigr\angler_{t^N_{j+1}}
\nonumber
\\
&&{} - Z^{N,x}_{t^N_{j+1}}\Delta\tilde{W}^N_{t^N_{j+1}}-
\Delta M^{N,x}_{t^N_{j+1}},
\\
Y^{N,x}_T &=& \varphi\bigl(w(s_1),
\ldots,w(s_r), x + \tilde{W}^N_{s_{r+1}}, \ldots,x +
\tilde{W}^N_{s_n}\bigr).
\nonumber
\end{eqnarray}
Thus, for $t\geq t^N_i$ we he have $Y^{N,x}_t =
\bar{Y}^{N,x}_{t}$, where $\bar{Y}^{N,x}$ solves the
BS$\Delta$E driven by the processes $W^N$ with terminal
conditions $\xi^{N,x} = \varphi(w(s_1),\ldots,w(s_r), x +
W^N_{s_{r+1}}-W^N_{t^N_i}, \ldots,x + W^N_{s_n}-W^N_{t^N_i})$
and drivers
\begin{eqnarray*}
&&\bar{f}^{N,x}\bigl(t,w\bigl(t^N_1\bigr),
\ldots,w\bigl(t^N_{i-1}\bigr),W^N,y,z\bigr)
\\
&&\quad = \cases{ f^N\bigl(t ,w\bigl(t^N_1\bigr),
\ldots,w\bigl(t^N_{i-1}\bigr), x + W^N-W^N_{t^N_i},
y,z\bigr), &\quad for $t > t^N_{i}$,
\cr
0, &\quad for $t \le
t^N_{i}$. }
\end{eqnarray*}
Clearly, all $\bar{f}^N$ are adapted, left-continuous and
satisfy (f1)--(f3). By our Lipschitz assumption on $\varphi$
and $f^N$, one has,
\[
\bigl\Vert\xi^{N,x_1} - \xi^{N,x_2}\bigr\Vert_{\infty} \le C
|x_1-x_2|
\]
and
\[
\bigl\Vert\bar{f}^{N,x_1} - \bar{f}^{N,x_2}\bigr\Vert_{\infty} \le K
|x_1 - x_2|
\]
for all $x_1,x_2 \in\mathbb{R}^d$. In particular,
\[
\bigl\Vert\bar{f}^{N,x_1} - \bar{f}^{N,x_2}\bigr\Vert_{\infty} \le K
\]
if $|x_1 - x_2| \le1$. So one obtains from Lemma
\ref{infterm2} that for all $x_1,x_2 \in\mathbb{R}^d$
satisfying $|x_1 - x_2| \le1$,
\begin{eqnarray*}
\bigl |Y^{N,x_1}_{t^N_i} - Y^{N,x_2}_{t^N_i}\bigr| &\leq&
\sup_{0\leq t\leq T}\bigl|\bar{Y}^{N,x_1}_t - \bar{Y}^{N,x_2}_t\bigr|
\\
&\le& D \bigl(\bigl\llVert \xi^{N,x_1}-\xi^{N,x_2} \bigr
\rrVert_\infty + \bigl\llVert \bar{f}^{N,x_1}-
\bar{f}^{N,x_2} \bigr\rrVert_\infty \bigr)
\\
& \le& D (C+K)|x_1 - x_2|.
\end{eqnarray*}
Note that
\[
{\mathbb{E}} \bigl[y^N_{t^N_i}\bigl(W^N_{t^N_1},
\ldots ,W^N_{t^N_{i-1}},W^N_{t^N_{i-1}}\bigr)
\Delta\bigl\anglel W^N\bigr\angler_{t^N_i} \mid
\F^N_{t^N_{i-1}} \bigr] = 0,
\]
and therefore,
\begin{eqnarray*}
\bigl|Z^{N,k}_{t^N_i}\bigr| &=& \Delta\bigl\anglel W^N\bigr
\angler_{t^N_i}^{-1} \bigl| {\mathbb{E}} \bigl[Y^N_{t^N_i}
\Delta W^{N,k}_{t^N_i}|\F^N_{t^N_{i-1}} \bigr] \bigr|
\\
&=& \bigl| {\mathbb{E}} \bigl[
\bigl(y^N_{t^N_i}\bigl(W^N_{t^N_1},\ldots,W^N_{t^N_{i-1}}, W^N_{t^N_{i-1}}+\Delta
W^{N}_{t^N_i}\bigr)\\
&&\quad\hspace*{7pt} {}- y^N_{t^N_i}\bigl(W^N_{t^N_1},\ldots
,W^N_{t^N_{i-1}},W^N_{t^N_{i-1}}\bigr) \bigr) \Delta W^{N,k}_{t^N_i} |\F^N_{t^N_{i-1}} \bigr] \bigr |/{\bigl(\Delta\bigl\anglel W^N\bigr\angler_{t^N_i}\bigr)}
\\
&\le&{\mathbb{E}} \bigl[\bigl|y^N_{t^N_i}\bigl(W^N_{t^N_1},\ldots
,W^N_{t^N_{i-1}}, W^N_{t^N_{i-1}}+\Delta
W^{N}_{t^N_i}\bigr)\\
&&\quad \hspace*{2pt}{} -y^N_{t^N_i}\bigl(W^N_{t^N_1},\ldots,W^N_{t^N_{i-1}},
W^N_{t^N_{i-1}}\bigr)\bigr|\bigl|\Delta W^{N,k}_{t^N_i}\bigr|  |\F^N_{t^N_{i-1}}
\bigr]/{\bigl(\Delta\bigl\anglel W^N\bigr\angler_{t^N_i}\bigr)}
\\
&\le& D(C+K) \frac{{\mathbb{E}} [ |\Delta W^{N}_{t^N_i} |
|\Delta W^{N,k}_{t^N_i}|  |\F^N_{t^N_{i-1}} ]}{\Delta\anglel W^N\angler_{t^N_i}}
\\
&\le& D(C+K) \frac{{\mathbb{E}} [ |\Delta W^{N}_{t^N_i} |
| \Delta W^{N}_{t^N_i}  | |\F^N_{t^N_{i-1}}
]}{\Delta\anglel W^N\angler_{t^N_i}}= D(C+K)d.
\end{eqnarray*}
\upqed\end{pf*}

%sA.4 #&#
\subsection{\texorpdfstring{Remaining proofs of Section \protect\ref{secvex}}
{Remaining proofs of Section 6}}

\begin{pf*}{Proof of Proposition~\ref{gN}}
Set $\bar{C} = (C+1) \exp(KT)$ and denote
\[
a = \sup_{N,i} \frac{\llVert |\Delta W^{N}_{t^N_i}| \rrVert_{\infty
}}{\sqrt {\Delta
\anglel W^{N}\angler_{t^N_i}}} < \infty.
\]
Choose $N_0 \in\mathbb{N}$ such that for all $N \ge N_0$ the
conclusion of Theorem~\ref{thmcomp} holds and
%
%eA.20 #&#
\begin{equation}
\label{N0} \sqrt{d} L a \bigl(\Delta\bigl\anglel W^{N}\bigr
\angler_{t^N_i}\bigr)^{1/2} + d^{(2+q)/4} L
\bar{C}^{q/2} a \bigl(\Delta\bigl\anglel W^{N}\bigr
\angler_{t^N_i}\bigr)^{(2-q)/4} < 1.
\end{equation}
Then it follows from Theorem~\ref{thmcomp} that for fixed $N
\ge N_0$, the $N$th BS$\Delta$E has a unique solution
$(Y^N,Z^N,M^N)$ with $|Y^N_t| \le\bar{C}$ for all $t \in
[0,T]$. Now choose an $\mathbb{R}^d$-valued $(\mathcal
{F}^N_t)$-adapted process
$\mu^N$ that is constant on the intervals $(t^N_{i-1}, t^N_i]$ and
satisfies \eqref{dens}. It
follows from the definition of $g^N$ that
\begin{eqnarray*}
Y^N_{t^N_i}&=& \xi^N + \sum
_{j=i+1}^{i_N}f^N\bigl(t^N_j,Y^N_{t^N_{j-1}},
Z^N_{t^N_j}\bigr)\Delta \bigl\anglel W^{N}\bigr
\angler_{t^N_j} - \sum_{j=i+1}^{i_N}
Z^N_{t^N_j}\Delta W^N_{t^N_j} -
\bigl(M^N_{T}-M^N_{t^N_i}\bigr)
\\
&\ge& \xi^N - \sum_{j=i+1}^{i_N}
g^N\bigl(t^N_j,Y^N_{t^N_{j-1}},
\mu^N_{t^N_j}\bigr)\Delta \bigl\anglel W^{N}\bigr
\angler_{t^N_j} -\sum_{j=i+1}^{i_N}
Z^N_{t^N_j} \Delta W^{N,\mu^N}_{t^N_j} -
\bigl(M^N_{T}-M^N_{t^N_i}\bigr).
\end{eqnarray*}
Since $M^N$ is orthogonal to $W^N$, its
components are still martingales under $P^{\mu^N}\!\!$, and one
obtains
%
%eA.21 #&#
\begin{equation}
\label{convineq} Y^N_{t^N_i} \ge\mathbb{E}^{\mu^N}
\Biggl[\xi^N- \sum_{j=i+1}^{i_N}
g^N\bigl(t^N_j,Y^N_{t^N_{j-1}},
\mu^N_{t^N_j}\bigr)\Delta\bigl\anglel W^{N}\bigr
\angler_{t^N_j} \Big|\F^N_{t^N_i} \Biggr].
\end{equation}
On the other hand, it
can be shown (see, e.g., Cheridito \textit{et al.}~\cite{8})
that for each $i$ there exists a $\hat{\mu}^N_{t^N_i} \in
L^0(\mathcal{F}^N_{t^N_{i-1}})^d$ such that
\[
f^N\bigl(t^N_i,Y^N_{t^N_{i-1}},Z^N_{t^N_i}
+ z\bigr) - f^N\bigl(t^N_i,Y^N_{t^N_{i-1}},Z^N_{t^N_i}
\bigr) \ge z \hat{\mu}^N_{t^N_i} \qquad\mbox{for all } z \in
\mathbb{R}^d.
\]
Set $\hat{\mu}^N_t = \hat{\mu}^N_{t^N_{i}}$ for $t \in
(t^N_{i-1}, t^N_i].$ Then $\hat{\mu}^N$ is a
left-continuous $\mathbb{R}^d$-valued $(\mathcal{F}^N_t)$-adapted
process satisfying
%
%eA.22 #&#
\begin{equation}
\label{Fen} f^N\bigl(t^N_i,Y^N_{t^N_{i-1}},Z^N_{t^N_i}
\bigr) + g^N\bigl(t^N_i,Y^N_{t^N_{i-1}},
\hat{\mu}^N_{t^N_i}\bigr) = \hat{\mu}^N_{t^N_i}
Z^N_{t^N_i} \qquad\mbox{for all } i.
\end{equation}
So if we can show that $\hat{\mu}^N$ satisfies \eqref{dens}
and \eqref{Repsilon}, the equality in \eqref{convineq} becomes
an equality and the proposition is proved. To see that
$\hat{\mu}^N$ satisfies \eqref{dens}, note that it follows
from the Cauchy--Schwarz inequality that
\begin{eqnarray*}
\bigl|Z^{N,k}_{t^N_i}\bigr| &=& \bigl|\bigl(\Delta\bigl\anglel W^{N}
\bigr\angler_{t^N_i}\bigr)^{-1} {\mathbb{E}}
\bigl[Y^N_{t^N_{i-1}} \Delta W^{N,k}_{t^N_i}|
\F^N_{t^N_{i-1}} \bigr]\bigr |
\\
&\le& \Bigl|\bigl(\Delta\bigl\anglel W^{N}\bigr\angler_{t^N_i}
\bigr)^{-1} \sqrt{{\mathbb{E}} \bigl[\bigl|Y^N_{t^N_{i-1}}\bigr|^2
\mid\mathcal {F}^N_{t^N_{i-1}} \bigr]} \sqrt{{
\mathbb{E}} \bigl[ \bigl|\Delta W^{N,k}_{t^N_i}\bigr |^2|
\F^N_{t^N_{i-1}} \bigr]} \Bigr|
\\
&\le& \bar{C} \bigl(\Delta\bigl\anglel W^{N}\bigr
\angler_{t^N_i}\bigr)^{-1/2},
\end{eqnarray*}
and
therefore,
%
%eA.23 #&#
\begin{equation}
\label{Zibound} \bigl|Z^N_{t^N_i}\bigr| \le\sqrt{d} \bar{C} \bigl(\Delta
\bigl\anglel W^{N}\bigr\angler_{t^N_i}\bigr)^{-1/2}.
\end{equation}
From condition
(v) one obtains
\[
\bigl|\hat{\mu}^{N,k}_{t^N_i}\bigr| \le L\bigl(1+\bigl|Z^N_{t^N_i}\bigr|^{q/2}
\bigr) \qquad\mbox{for all } k.
\]
Hence, it follows from estimate \eqref{Zibound} that
\[
\bigl|\hat{\mu}^N_{t^N_i}\bigr| \le\sqrt{d} L\bigl(1+\bigl|Z^N_{t^N_i}\bigr|^{q/2}
\bigr) \le \sqrt{d} L + d^{(2+q)/4} L \bar{C}^{q/2} \bigl(\Delta
\bigl\anglel W^{N}\bigr\angler_{t^N_i}\bigr)^{-q/4}.
\]
This gives
\[
\bigl|\hat{\mu}^{N}_{t^N_{i}} \Delta W^{N}_{t^N_{i}}\bigr|
\le\bigl|\hat{\mu}^N_{t^N_i}\bigr|\bigl |\Delta W^{N}_{t^N_{i}}\bigr|
\le\sqrt{d} L a \bigl(\Delta\bigl\anglel W^{N}\bigr
\angler_{t^N_i}\bigr)^{1/2} + d^{(2+q)/4} L
\bar{C}^{q/2} a \bigl(\Delta\bigl\anglel W^{N}\bigr
\angler_{t^N_i}\bigr)^{(2-q)/4} < 1
\]
and shows that $\hat{\mu}^N$ satisfies condition
\eqref{dens}.

To show \eqref{Repsilon}, we first assume $q=1$. Then one has
\begin{eqnarray*}
 g^N\bigl(t^N_{j},Y^N_{t^N_{j-1}},
\hat{\mu}^N_{t^N_{j}}\bigr) &=& \mathop{\ess\sup}_z \bigl\{\hat{\mu}^N_{t^N_{j}}
z-f^N\bigl(t^N_j,Y^N_{t^N_{j-1}},z
\bigr)\bigr\}
\\
&\ge& \mathop{\ess\sup}_z \bigl\{\hat{\mu}^N_{t^N_{j}}
z-K \bigl(1+\bigl|Y^N_{t^N_{j-1}}\bigr|+|z|\bigr)\bigr\}.
\end{eqnarray*}
It follows that
\[
\bigl|\hat{\mu}^{N,k}_{t^N_j}\bigr| \le K \qquad\mbox{for all } k =1, \ldots, d,
\]
and it is clear that $\hat{\mu}^N$ satisfies condition
\eqref{Repsilon}. If $q \in(1,2)$, denote $|x|_q =
(\sum_{i=1}^d |x_i|^q)^{1/q}$, and observe that there exist
constants $C_1,C_2,C_3 > 0$ such that
%
%eA.24 #&#
\begin{eqnarray}\label{gmu}
g^N\bigl(t^N_{j},Y^N_{t^N_{j-1}},
\hat{\mu}^N_{t^N_{j}}\bigr) &=& \mathop{\ess\sup}_z \bigl\{\hat{\mu}^N_{t^N_{j}}
z-f^N\bigl(t^N_{j+1},Y^N_{t^N_{j}},z
\bigr)\bigr\}\nonumber
\\
&\ge& \mathop{\ess\sup}_z \bigl\{\hat{
\mu}^N_{t^N_{j}} z - K \bigl(1+\bigl|Y^N_{t^N_{j}}\bigr|+|z|^q
\bigr)\bigr\}\nonumber\\[-8pt]\\[-8pt]
 &\ge& -K \bigl(1+\bigl|Y^N_{t^N_{j}}\bigr|\bigr)+ \mathop{\ess\sup}_z \bigl\{\hat{\mu}^N_{t^N_{j}} z -
C_1 |z|_q^q\bigr\}\nonumber\\
&=& -K \bigl(1+\bigl|Y^N_{t^N_{j}}\bigr|\bigr) +
C_2\bigl|\hat{\mu}^N_{t^N_{j}}\bigr|^{q/(q-1)}_{q/(q-1)}
\ge -K\bigl(1+\bigl|Y^N_{t^N_{j}}\bigr|\bigr) + C_3\bigl(|
\hat{\mu}_{t^N_{j}}|^2+1\bigr).\nonumber
\end{eqnarray}
Since
\[
Y^N_{t^N_i} = \mathbb{E}^{\hat{\mu}^N} \Biggl[
\xi^N- \sum_{j=i+1}^{i_N}
g^N\bigl(t^N_j,Y^N_{t^N_{j-1}},
\hat{\mu }^N_{t^N_j}\bigr) \Delta\bigl\anglel W^{N}
\bigr\angler_{t^N_j} |\F^N_{t^N_i} \Biggr]
\]
and $\xi^N$ and $Y^N_t$ are bounded by $C$ and $\bar{C}$,
respectively, one obtains
\[
\mathbb{E}^{\hat{\mu}^N} \Biggl[\sum_{j=i+1}^{i_N}
g^N\bigl(t^N_j,Y^N_{t^N_{j-1}},
\hat{\mu }^N_{t^N_j}\bigr) \Delta\bigl\anglel W^{N}
\bigr\angler_{t^N_j} \mid\mathcal {F}^N_{t^N_i} \Biggr]
\le C + \bar{C}.
\]
This together with \eqref{gmu} and the uniform boundedness of
$Y^N$ shows that $\hat{\mu}^N$ fulfills~\eqref{Repsilon}.
\end{pf*}

%leA.3 #&#
\begin{lemma}
\label{2mubound} Let $\mu$ be an $(\mathcal{F}^N_t)$-adapted process
that is constant on the intervals $(t^N_{i-1}, t^N_i]$
and satisfies \eqref{dens}. Then one has
\begin{eqnarray*}
&&{\mathbb{E}} \Biggl[\prod_{j = i+1}^{i_N}
\bigl(1 + \mu_{t^N_j} \Delta W^N_{t^N_j}\bigr) \log
\Biggl(\prod_{j = i+1}^{i_N} \bigl(1 +
\mu_{t^N_j} \Delta W^N_{t^N_j}\bigr) \Biggr) \Big|
\F^N_{t^N_i} \Biggr] \\
&&\quad\le {\mathbb{E}}^\mu \Biggl[
\sum_{j=i+1}^{i_N} |\mu_{t^N_{j}}|^{2}
\Delta\bigl\anglel W^N\bigr\angler_{t^N_{j}} |
\F^N_{t^N_i} \Biggr].
\end{eqnarray*}
\end{lemma}

\begin{pf}
One can write
\begin{eqnarray*}
&& {\mathbb{E}} \Biggl[\prod_{j = i+1}^{i_N}
\bigl(1 + \mu_{t^N_j} \Delta W^N_{t^N_j}\bigr) \log
\Biggl(\prod_{j = i+1}^{i_N} \bigl(1 +
\mu_{t^N_j} \Delta W^N_{t^N_j}\bigr) \Biggr) \Big|
\F^N_{t^N_i} \Biggr]
\\[-2pt]
&&\quad= \sum_{j = i+1}^{i_N} {\mathbb{E}}^\mu
\bigl[\log\bigl(1 + \mu_{t^N_j} \Delta W^N_{t^N_j}
\bigr) | \F^N_{t^N_i} \bigr] \le\sum
_{j = i+1}^{i_N} \log \bigl( {\mathbb{E}}^\mu
\bigl[\bigl(1 + \mu_{t^N_j} \Delta W^N_{t^N_j}\bigr) |
\F^N_{t^N_i} \bigr] \bigr),
\end{eqnarray*}
where the inequality follows from
Jensen's inequality. The right-hand side can be estimated as
follows:
\begin{eqnarray*}
&& \sum_{j = i+1}^{i_N} \log \Biggl\{1+\sum
_{k=1}^d {\mathbb{E}}^\mu
\bigl[\mu^{k}_{t^N_j} {\mathbb {E}}^\mu \bigl[\Delta
W^{N,k}_{t^N_j} |\F^N_{t^N_{j-1}} \bigr] |
\F^N_{t^N_i} \bigr] \Biggr\}
\\[-2pt]
&&\quad= \sum_{j=i+1}^{i_N}\log \Biggl\{1+\sum
_{k=1}^d {\mathbb{E}}^\mu
\bigl[\bigl(\mu^{k}_{t^N_{j}}\bigr)^2 \Delta\bigl
\anglel W^{N}\bigr\angler_{t^N_{j}}|\F^N_{t^N_i}
\bigr] \Biggr\} \le\sum_{j =
i+1}^{i_N} {
\mathbb{E}}^\mu \bigl[|\mu_{t^N_{j}}|^2\Delta\bigl
\anglel W^N\bigr\angler_{t^N_{j}}|\F^N_{t^N_i}
\bigr].
\end{eqnarray*}
The equality holds
because
\[
{\mathbb{E}}^\mu \bigl[\Delta W^{N,k}_{t^N_{j}}|
\F^N_{t^N_{j-1}} \bigr] = \mu^k_{t^N_j} \Delta
\bigl\anglel W^N_{t^N_j}\bigr\angler .
\]
For the inequality we used $\log(1+x)\leq x$.
\end{pf}

%leA.4 #&#
\begin{lemma} \label{lemmagronwall}
For all $N \in\mathbb{N}$, let $h^N \dvtx  [0,T] \to\mathbb{R}$ be a function
that is constant on the intervals $[t^N_i, t^N_{i+1})$. If there exist
constants $a,b \in
\mathbb{R}_+$ such that
\[
\bigl|h^N(T)\bigr| \le a \quad\mbox{and}\quad \bigl|h^N\bigl(t^N_i
\bigr)\bigr| \le a + b \sum_{j=i+1}^{i_N}
\bigl|h^N\bigl(t^N_{j-1}\bigr)\bigr| \Delta\bigl\anglel
W^N\bigr\angler_{t^N_{j}} \qquad\mbox{for all } N \mbox{ and } i \le
i_N - 1,
\]
there exists an $N_0 \in\mathbb{N}$ such that
\[
\bigl|h^N\bigl(t^N_i\bigr)\bigr| \le2a\exp\bigl(b
\bigl(T-t^N_i\bigr)\bigr)\qquad \mbox{for all } N \ge
N_0 \mbox{ and } i =0 , \ldots, i_N.
\]
\end{lemma}

\begin{pf}
For $N$ so large that $\sup_i \Delta\anglel W^N\angler_{t^N_i} <
1/b$, the
function given by
\[
H^N(T) :=a \quad\mbox{and}\quad H^N(t) := a \prod
_{j : t^N_j > t} \bigl(1 - b \Delta\bigl\anglel W^N\bigr
\angler_{t^N_j} \bigr)^{-1},\qquad t < T
\]
solves
\[
H^N\bigl(t^N_i\bigr) = a + b \sum
_{j=i+1}^{i_N} H^N\bigl(t^N_{j-1}
\bigr) \Delta\bigl\anglel W^N\bigr\angler_{t^N_{j}}\qquad \mbox{for
all } i \le i_N - 1,
\]
and converges uniformly to $a \exp(b(T\!-\!t))$. In particular, there
exists an $N_0\!\in\!\mathbb{N}$ such that
\[
H^N(t) \le2 a \exp\bigl(B(T-t)\bigr)\qquad \mbox{for all $t$ and $N \ge
N_0$}.\vadjust{\goodbreak}
\]
So the lemma follows if we can show
that $|h^N(t^N_i)| \le H^N(t^N_i)$ for all $N \ge N_0$ and $i = 0,
\ldots, i_N$.
For $i = i_N$ this is obvious, and if it holds for $j \ge i+1$, then
\begin{eqnarray*}
\bigl|h^N\bigl(t^N_i\bigr)\bigr| &\le&
\frac{a + b \sum_{j=i+2}^{i_N} |h^N(t^N_{j-1})|
\Delta\anglel W^N\angler_{t^N_j}}{1 - b \Delta\anglel W^N\angler_{t^N_{i+1}}}
\\
&\le& \frac{a + b \sum_{j=i+2}^{i_N} |h^N(t^N_{j-1})| \Delta\anglel W^N\angler_{t^N_j}}{1 - b \Delta\anglel W^N\angler_{t^N_{i+1}}} = H^N\bigl(t^N_i
\bigr).
\end{eqnarray*}
\upqed\end{pf}

\end{appendix}

\section*{Acknowledgments}
Financial support from NSF Grant DMS-06-42361 is gratefully
acknowledged.

%suskaldyti doi

% imsref loaded by aiste.veprauskaite, 2012-08-06 11:15:22
% imsref loaded by aiste.veprauskaite, 2012-08-06 11:22:58

% imsref loaded by aiste.veprauskaite, 2012-08-06 13:10:06
% imsref loaded by aiste.veprauskaite, 2012-08-06 13:24:56
% imsref loaded by aiste.veprauskaite, 2012-08-06 13:37:32

\printhistory

\end{document}